\documentclass[11pt]{article}
\usepackage{times}
\usepackage{amssymb,amsmath}
\usepackage{epsf}
\usepackage{times,bm,mathrsfs}  
\usepackage{amsmath}
\usepackage{amssymb}
\usepackage{amsfonts}
\usepackage{mathrsfs}
\usepackage{bbm}
\usepackage{color}
\usepackage{graphicx}
\usepackage{amscd}
\usepackage{epsfig}

\definecolor{Red}{rgb}{1,0,0}
\definecolor{Blue}{rgb}{0,0,1}
\definecolor{Olive}{rgb}{0.41,0.55,0.13}
\definecolor{Green}{rgb}{0,1,0}
\definecolor{MGreen}{rgb}{0,0.8,0}
\definecolor{DGreen}{rgb}{0,0.55,0}
\definecolor{Yellow}{rgb}{1,1,0}
\definecolor{Cyan}{rgb}{0,1,1}
\definecolor{Magenta}{rgb}{1,0,1}
\definecolor{Orange}{rgb}{1,.5,0}
\definecolor{Violet}{rgb}{.5,0,.5}
\definecolor{Purple}{rgb}{.75,0,.25}
\definecolor{Brown}{rgb}{.75,.5,.25}
\definecolor{Grey}{rgb}{.5,.5,.5}
\definecolor{Black}{rgb}{0,0,0}

\def\path{{\tt path}}
\def\lam{{\lambda}}

\newcommand{\eps}{\varepsilon}

\newcommand{\bdm}{\begin{displaymath}}
\newcommand{\edm}{\end{displaymath}}
\newcommand{\bea}{\begin{eqnarray*}}
\newcommand{\eea}{\end{eqnarray*}}
\newcommand{\bean}{\begin{eqnarray}}
\newcommand{\eean}{\end{eqnarray}}

\newcommand{\polylog}{\mathrm{polylog}}

\newtheorem{theorem}{Theorem}
\newtheorem{proposition}{Proposition}
\newtheorem{corollary}{Corollary}
\newtheorem{definition}{Definition}

\newtheorem{lemma}{Lemma}

\newtheorem{claim}{Claim}[section]

\newenvironment{proof}{\noindent{\textbf{Proof:}}}{$\blacksquare$\vskip\belowdisplayskip}

\title{Gibbs Rapidly Samples Colorings of $G(n,d/n)$}

\author{Elchanan Mossel\thanks{Email: mossel@stat.berkeley.edu. Dept. of Statistics, U.C. Berkeley. Supported by an Alfred Sloan fellowship
  in Mathematics and by NSF grants DMS-0528488,
  DMS-0548249 (CAREER) by  DOD ONR grant N0014-07-1-05-06.} \and
  Allan Sly \thanks{{Email: sly@stat.berkeley.edu Dept. of Statistics, U.C. Berkeley. Supported by NSF grants DMS-0528488 and DMS-0548249 }}}
\begin{document}

\maketitle

\thispagestyle{empty}

\begin{abstract}

Gibbs sampling also known as Glauber dynamics is a popular technique
for sampling high dimensional distributions defined on graphs. Of
special interest is the behavior of Gibbs sampling on the
Erd\H{o}s-R\'enyi random graph $G(n,d/n)$, where each edge is chosen
independently with probability $d/n$ and $d$ is fixed. While the
average degree in $G(n,d/n)$ is $d(1-o(1))$, it contains many nodes
of degree of order $(\log n) / (\log \log n)$.

The existence of nodes of almost logarithmic degrees implies that
for many natural distributions defined on $G(n,d/n)$ such as uniform
coloring (with a constant number of colors) or the Ising model at
any fixed inverse temperature $\beta$, the mixing time of Gibbs
sampling is at least $n^{1 + \Omega(1 / \log \log n)}$ with high
probability.
High degree nodes
pose a technical challenge in proving polynomial time mixing of the dynamics
for many models including coloring.
Almost all known sufficient conditions in terms of number of colors needed
for rapid mixing of Gibbs samplers are stated in terms
of the maximum degree of the underlying graph.

In this work consider sampling $q$-colorings and show that for every
$d < \infty$ there exists $q(d) < \infty$ such that for all $q \geq
q(d)$ the mixing time of Gibbs sampling on $G(n,d/n)$ is polynomial
in $n$ with high probability. Our results are the first polynomial
time mixing results proven for the coloring model on $G(n,d/n)$ for
$d > 1$ where the number of colors does not depend on $n$. They also
provide a rare example where one can prove a polynomial time mixing
of Gibbs sampler in a situation where the actual mixing time is
slower than $n \polylog(n)$. In previous work we have shown that
similar results hold for the ferromagnetic Ising model. However, the
proof for the Ising model crucially relied on monotonicity arguments
and the ``Weitz tree'' both of which have no counterparts in the
coloring setting. Our proof presented here exploits in novel ways
the local treelike structure of Erd\H{o}s-R\'enyi random graphs,
block dynamics, spatial decay properties and coupling arguments.

Our results give first FPRAS to sample coloring on $G(n,d/n)$ with a
constant number of colors. They extend to much more general families
of graphs which are sparse in some average sense and to much more
general interactions. In particular, they apply to any graph for
which there exists an $\alpha > 0$ such that every vertex $v$ of the
graph has a neighborhood $N(v)$ of radius $O(\log n)$ in which the
induced sub-graph is the union of a tree and at most $O(1)$ edges
and where each simple path $\Gamma$ of length $O(\log n)$ satisfies
$\sum_{u \in \Gamma} \sum_{v \neq u} \alpha^{d(u,v)} = O(\log n)$.
The results also generalize to the hard-core model at low fugacity
and to general models of soft constraints at high temperatures.

\end{abstract}

\bigskip


\clearpage

\section{Introduction}
Efficient approximate sampling from Gibbs distributions is a central
challenge of randomized algorithms. Examples include sampling from
the uniform distribution over independent sets of a
graph~\cite{Weitz:06,Vigoda:01,DyFrJe:99,DyerGreenhill:97}, sampling
from the uniform distribution of matchings in a
graph~\cite{JeSiVi:04}, or sampling from the uniform distribution of
colorings~\cite{GoMaPa:04,DFHV:04,DFFV:06} of a graph. A natural
family of approximate sampling techniques is given by Gibbs
samplers, also known as Glauber dynamics. These are reversible
Markov chains that have the desired distribution as their stationary
distribution and where at each step the status of one vertex is
updated. It is typically easy to establish that the chains
will eventually converge to the desired distribution.\\

Studying the convergence rate of the dynamics is interesting from
both the theoretical computer science and the statistical physics
perspectives. Approximate convergence in polynomial time, sometimes
called {\em rapid mixing}, is essential in computer science
applications. The convergence rate is also of natural interest in
the physics where the dynamical properties of such distributions are
extensively studied, see e.g.~\cite{Martinelli:99}. Much recent work
has been devoted to determining sufficient and necessary conditions
for rapid convergence of Gibbs samplers. A common feature to most of
this
work~\cite{Weitz:06,Vigoda:01,DyFrJe:99,DyerGreenhill:97,GoMaPa:04,DFHV:04,KeMoPe:01,MaSiWe:03a}
is that the conditions for convergence are stated in terms of the
maximal degree of the underlying graph. In particular, these results
do not allow for the analysis of the mixing rate of Gibbs samplers
on the Erd\H{o}s-R\'enyi random graph, which is sparse on average,
but has a small number of denser sub-graphs.
In a recent work~\cite{MosselSly:07} we have shown that for any $d$ if
$0 \leq \beta < \beta(d)$ is sufficiently small then Gibbs sampling for the
{\em Ising model} on on $G(n,d/n)$ rapidly mixes. We show that the same result
is true in the presence of arbitrary {\em external field}.
The proofs of~\cite{MosselSly:07} crucially rely on the monotonicity of the
Ising model and on the {\em ``Weitz tree''}~\cite{Weitz:06} which is
only defined for two spin models. Thus the proof does not apply to models
such as the {\em hard-core model} or to sampling {\em uniform coloring}.
Other recent work has been invested in showing
how to relax statements so that they do not involve maximal
degrees~\cite{DFFV:06,Hayes:06}, but the results are not strong
enough to imply rapid mixing of Gibbs sampling for uniform colorings on
$G(n,d/n)$ for $d > 1$ and $O(1)$ colors. This
is presented as a major open problem of both ~\cite{DFFV:06}
and~\cite{MosselSly:07}.\\

In this paper we give the first rapid convergence result of Gibbs samplers
for the Ising model on Erd\H{o}s-R\'enyi random graphs in terms of the average degree and the number of colors only. Our results yields the first
FPRAS for sampling the coloring distribution in this case.
Our results are further extended to more general families of
graphs that are ``tree-like'' and ``sparse on average''. These are
graph where every vertex has a radius $O(\log n)$ neighborhood which
is a tree with at most $O(1)$ edges added and where for each
simple path $\Gamma$ of length $O(\log n)$ it holds that
$\sum_{u \in \Gamma} \sum_{v \neq u} \alpha^{d(u,v)} \leq O(\log n)$, where
$\alpha > 0$ is some fixed parameter.

Subsequent to completing this work we learned that Spirakis and
Efthymiou \cite{EfthymioSpirakis:07} independently have also produced a
scheme for approximately sampling from the random coloring
distribution in polynomial time. They take a different approach,
instead of sampling using MCMC they assign colours to vertices one
at a time by calculating the conditional marginal distributions
making use of the decay in correlation on the graph.

Our arguments extend to prove similar results for many other models.
In particular, they give an independent proof of rapid mixing for
sampling from the Ising model for small inverse temperature $\beta$,
the hard-core model for small fugacity $\lambda$ and many other
models. Note however, that the result presented here for the Ising
model on general graphs are slightly weaker than the result
of~\cite{MosselSly:07}. Here we require that each $O(\log n)$ radius
neighborhood is a tree union a constant number of edges while
in~\cite{MosselSly:07} an excess of $O(\log n)$ is allowed.

Below we define the coloring model and Gibbs samplers and state our main result
for coloring.
Some related work and a sketch of the proof are also given as the introduction.
Section~\ref{sec:proofs} gives a more detailed proof.

\subsection{Models}
Our results cover a wide range of graph based distributions
including the coloring model, the hardcore model and any model with
soft constraints.

\begin{definition} \label{def:models}
Let $G=(V,E)$ be a graph and let $\mathcal{C}$ be a set of
states/colours with $|\mathcal{C}|=q$. The Hamiltonian is a function
$\mathcal{C}^V\rightarrow \mathbb{R}$ of the form
\begin{equation} \label{eq:defH}
H(\sigma)=\sum_{u\in V} h(\sigma(u))+\sum_{(u,v)\in E}
g(\sigma(u),\sigma(v))
\end{equation}
where $h:\mathcal{C}\rightarrow \mathbb{R}$ is the activity function
and $g:\mathcal{C}^2\rightarrow \mathbb{R}\cup\{-\infty\}$ is a
symmetric interaction function. This defines an interacting particle
system which is the distribution on $\sigma\in\mathcal{C}^V$ given
by
\[
P(\sigma)=\frac1{Z} \exp(H(\sigma))
\]
where $Z$ is a normalizing constant.  We focus our attention on 3
classes of models.
\begin{itemize}
\item The {\em coloring distribution} is the uniform distribution over
colorings of $G$ with $h\equiv 0$ and $g(x,y)=-\infty 1_{\{x=y\}}$
so the distribution is given by
\begin{equation} \label{eq:def_coloring}
P(\sigma)= \frac1{Z}\prod_{(u,v)\in E} 1_{\{\sigma(u)\neq
\sigma(v)\}}.
\end{equation}
\item
The {\em hardcore} model with parameter $\beta$ is the weighted
distribution over independents sets of $G$ given by
$\mathcal{C}=\{0,1\}$ with $h(x)=\beta x$ and $g(x,y)=-\infty
1_{\{x=y=1\}}$ and
\begin{equation} \label{eq:def_hardcore}
P(\sigma)= \frac1{Z}\exp(\beta \sum_{u\in V}\sigma(u)
)\prod_{(u,v)\in E} 1_{\{\sigma(u) \sigma(v)=0\}}
\end{equation}
where $\sigma$ takes values in $\{0,1\}^V$ and $Z$ is a normalizing
constant.
\item
If $g$ does not take the value $-\infty$ then we say the model has
soft-constraints. This class includes the Ising model.
\end{itemize}
For $U\subset V$ we let $P_U$ be the colouring model on the subgraph
induced by $U$.  Define the {\em activity free system} $\widehat{P}$
as the distribution with the activity function $h$ set to 0. The
norm of the Hamiltonian is defined
\[
\| H\|:=\max\left\{ \max_{x\in \mathcal{C}} |h(x)|, \max_{x,y\in
\mathcal{C}} |g(x,y)| \right\}.
\]
\end{definition}

\subsection{Gibbs Sampling}

The Gibbs sampler is a Markov chain on configurations where a
configuration $\sigma$ is updated by choosing a vertex $v$ uniformly
at random and assigning it a spin according to the Gibbs
distribution conditional on the spins on $G-\{v\}$.
\begin{definition} \label{def:gibbs}
Given a graph $G=(V,E)$, a set $\mathcal{C}$ and a Hamiltonian $H$
as in~(\ref{eq:defH}), the Gibbs sampler is the discrete time Markov
chain on $\mathcal{C}^V$ where given the current configuration
$\sigma$ the next configuration $\sigma'$ is obtained by choosing a
vertex $v$ in $V$ uniformly at random and
\begin{itemize}
\item
Letting $\sigma'(w) = \sigma(w)$ for all $w \neq v$.
\item
$\sigma'(v)$ is assigned the element $x \in \mathcal{X}$ with probability
proportional to
\[
\frac{1}{Z'} \exp\left( h(x) + \sum_{w \in N(v)} g(\sigma(w),x) \right).
\]
 where $N(v) = \{w \in V : (v,w) \in E\}$ and $Z'$ is a normalization constant.

Note that in the case of coloring $\sigma'(v)$ is chosen uniformly from
the set $\mathcal{C} \setminus \{\sigma(w) : w \in N(v)\}$.
\end{itemize}
\end{definition}
In the coloring model, it is not completely trivial to find an
initial configuration that is a legal coloring. However, for
$G(n,d/n)$ finding an initial coloring is
easy~\cite{ShamirUpfal:84}. It is well known that with high
probability if one removes all nodes of large enough degree $D'(d)$
from $G(n,d/n)$ then what remains is a collection of unicyclic
components. It is easy to color each unicyclic component with $3$
colors and therefore color the graph with $D'+3$ colors. Similar
arguments will allow us to find an initial coloring in the more
general setting discussed here. See~\cite{FriezeMcDiarmid:97} for a
survey of algorithmic results for finding legal coloring in sparse
random graphs. For the hard-core model and models with soft
constraints, it is trivial to find an initial legal configuration.

We will be interested in the time it takes the dynamics to get close
to the distributions~(\ref{eq:def_coloring}).
The {\em mixing time } $\tau_{mix}$ of the chain is defined as the
number of steps needed in order to guarantee that the chain,
starting from an arbitrary state, is within total variation distance
$(2e)^{-1}$ from the stationary distribution.

\subsection{Erd\H{o}s-R\'enyi Random Graphs and Other Models of graphs}

The Erd\H{o}s-R\'enyi random graph $G(n,p)$, is the graph with $n$ vertices
$V$ and random edges $E$ where each potential edge $(u,v) \in V \times V$
is chosen independently with probability $p$. We take
$p=d/n$ where $d \geq 1$ is fixed. In the case $d < 1$, it is well known that
with high probability all components of $G(n,p)$ are unicyclic
and of logarithmic size which implies immediately that the
dynamics considered here mix in time polynomial in $n$.

For a vertex $v$ in $G(n,d/n)$ let $V(v,l)=\{u\in G: d(u,v)\leq
l\}$, the set of vertices within distance $l$ of $v$, let
$S(v,l)=\{u\in G:d(u,v)=l\}$, let $E(v,l)=\{ (u,w)\in G: u,w \in
V(v,l) \}$ and let $B(v,l)$ be the graph ($V(v,l),E(v,l))$.

Our results only require some simple features of the neighborhoods of all vertices in the graph stated in terms of $t$ and $m$ below.
\begin{definition}
Let $G = (V,E)$ be a graph and $v$ a vertex in $G$.
Let $t(G)$ denote the {\em tree access} of $G$, i.e.,
\[
t(G) = |E| - |V| + 1.
\]
For $v \in V$ we let
$t(v,l) = t(B(v,l))$.

We call a path $v_1,v_2,\ldots$ {\em self avoiding} if for all
$i\neq j$ it holds that $v_i \neq v_{j}$.

For $\alpha>0$ we let the {\em maximal path $\alpha$-weight}
$m_{\alpha}$ of a subgraph $H \subset G$ be defined by
\[
m_{\alpha}(H,l) = \max_{\Gamma} \sum_{u \in \Gamma} \sum_{v : u \neq v \in G} \alpha^{d(u,v)}
\]
where the maximum is taken
over all self-avoiding paths $\Gamma \subset H$ of length at most $l$.
\end{definition}

\subsection{Our Results}
\subsubsection{Colouring Model}

\begin{theorem}\label{t:randomGraph}
For all $d \geq 1$ there exists $q(d) < \infty$ such that for all
$q \geq q(d)$ the following holds.
Let $G$ be a random graph distributed as $G(n,d/n)$.  Then
with high probability
the mixing time of Gibbs sampling of $q$-colorings is $O(n^{C})$.
\end{theorem}

The theorem above may be viewed as a special case of the more general result.
\begin{theorem}\label{t:main}
For any $0 < a,\alpha, t,\delta < \infty$ there exists constants
$q(a,\alpha, t,\delta)$ and $C=C(a,\alpha, t,\delta)$ such that if
$q \geq q(a,\alpha,t,\delta)$ and
$G=(V,E)$ is any graph on $n$ vertices satisfying
\begin{equation}\label{eq_thm_hypothesis}
\forall v \in V, t(v,a \log n)\leq t, \quad m_{\alpha}(G,a \log n)< \delta \log
n,
\end{equation}
then the mixing time of the Gibbs-sampler of $q$-colorings of $G$ is
$O(n^{C})$.
\end{theorem}

\subsubsection{Hardcore Model}

\begin{theorem}\label{t:randomGraphHard}
For all $d \geq 1$ there exists $\beta(d) < \infty$ such that for
all $\beta \leq \beta(d)$ the following holds. Let $G$ be a random
graph distributed as $G(n,d/n)$.  Then with high probability the
mixing time of Gibbs sampling of the hardcore model with parameter
$\beta$ is $O(n^{C})$.
\end{theorem}

The theorem above may be viewed as a special case of the more
general result.
\begin{theorem}\label{t:mainHard}
For any $0 < a,\alpha, t,\delta < \infty$ there exists constants
$\beta(a,\alpha, t,\delta)$ and $C=C(a,\alpha, t,\delta)$ such that
if $\beta \leq \beta(a,\alpha,t,\delta)$ and $G=(V,E)$ is any graph
on $n$ vertices satisfying
\begin{equation}\label{eq_thm_hypothesis_hard}
\forall v \in V, t(v,a \log n)\leq t, \quad m_{\alpha}(G,a \log n)<
\delta \log n,
\end{equation}
then the mixing time of the Gibbs-sampler of the hardcore model with
parameter $\beta$ is $O(n^{C})$.
\end{theorem}

\subsubsection{Soft Constraints}

\begin{theorem}\label{t:randomGraphSoft}
For all $d \geq 1$ there exists $0 < H^*(d) < \infty$ such that for all
models with $\|H\| \leq H^*(d)$ the following holds. Let $G$ be a
random graph distributed as $G(n,d/n)$.  Then with high probability
the mixing time of Gibbs sampling of the model is $O(n^{C})$.
\end{theorem}

The theorem above may be viewed as a special case of the more
general result.
\begin{theorem}\label{t:mainSoft}
For any $0 < a,\alpha, t,\delta < \infty$ and all soft constraint
models there exists constants $H^*(a,\alpha, t,\delta) > 0$ and
$C=C(a,\alpha, t,\delta)$ such that if $\|H\| \leq
H^*(a,\alpha,t,\delta)$ and $G=(V,E)$ is any graph on $n$ vertices
satisfying
\begin{equation}\label{eq_thm_hypothesis_soft}
\forall v \in V, t(v,a \log n)\leq t, \quad m_{\alpha}(G,a \log n)<
\delta \log n,
\end{equation}
then the mixing time of the Gibbs-sampler of the model is
$O(n^{C})$.
\end{theorem}

\subsection{Related Work}
Most results for mixing rates of Gibbs samplers are stated in terms
of the maximal degree. Thus for sampling uniform colorings, the
result are of the form: for every graph where all degrees are at
most $d$ if the number of colors $q$ satisfies $q \geq q(d)$ then
Gibbs sampling is rapidly
mixing~\cite{Weitz:06,Vigoda:01,DyFrJe:99,DyerGreenhill:97,GoMaPa:04,DFHV:04,KeMoPe:01,MaSiWe:03a}.
For example, it is well known and easy to see that one can take
$q(d) = 2d$. Similarly, results for the Ising model are stated in
terms of $\beta < \beta(d)$. The novelty of the result
of~\cite{MosselSly:07} and the result presented here is that it
allows us to study graphs where the average degree is small while
some degrees may be large.

Previous attempts at studying this problem for sampling uniform
colorings yielded weaker results. In~\cite{DFFV:06} it is shown that
Gibbs sampling rapidly mixes on $G(n,d/n)$ if $q = \Omega_d((\log
n)^{\alpha})$ where $\alpha < 1$ and that a variant of the algorithm
rapidly mixes if $q \geq \Omega_d(\log \log n / \log \log \log n)$.
Indeed the main open problem of~\cite{DFFV:06} is to determine if
one can take $q$ to be a function of $d$ only.\\

Comparing the results presented here to~\cite{MosselSly:07} we
observe first that there is one sense in which the current results
are weaker. In~\cite{MosselSly:07} the tree access $t$ can be of
order $O(\log n)$ while for the results presented here $t$ has to be
of order $O(1)$. The results of~\cite{MosselSly:07} crucially use
the fact that the Ising model is attractive (this is a monotonicity
property) and that it is a two spin system which allows using the
``Weitz tree''~\cite{Weitz:06}.

We note that for all $q$ and all $d$ the mixing time of Gibbs sampling
on $G(n,d/n)$ is with high probability at least
$n^{1+\Omega(1/\log \log n)} >> n \polylog(n)$, see~\cite{DFFV:06,MosselSly:07}
for details. It is an important challenge to find the critical $q=q(d)$
for rapid mixing. In particular, the question is if the threshold
can be formulated in terms of the coloring model on a branching process
tree with $Poisson(d)$ degree distribution. One would expect rapid mixing
for in the ``uniqueness phase'', but perhaps even beyond it,
see~\cite{MoWeWo:07,MosselSly:07,GerschenfeldMontanari:07}.

\subsection{Proof Technique}
We briefly sketch the main ideas behind the proof focusing on the
special case of coloring.
\paragraph{Block Dynamics and Path Coupling.}
The basic idea of the proof is quite standard. It is based on a combination
of {\em block dynamics}, see e.g.~\cite{Martinelli:99},
and {\em path coupling}, see e.g.~\cite{BubleyDyer:97}, techniques.
We wish to divide the
vertex set $V$ of the graph $G$ into disjoint blocks $V_1,\ldots,V_K$
with the following properties:
\begin{itemize}
\item
There is at most one edge between any pair of blocks.

\item For each block $V_i$ and any boundary conditions outside the
block, the relaxation time of the dynamics restricted to $V_i$ is polynomial
in $n$.

\item If we consider the block dynamics, where we pick a vertex $v \in V$
uniformly at random and update the block $V_i$ containing it according to the
conditional probability on $V \setminus V_i$, then it has the following
property: Given two configurations $\sigma$ and $\tau$ that differ at
one vertex $v$, the updated configurations $\sigma'$ and $\tau'$ may be coupled
is such a way that the expected number of differences between them
is $1-\Theta(1/n)$.
\end{itemize}
The properties above imply a polynomial mixing time for
the single site Gibbs-sampling dynamics.

\paragraph{Block Decomposition : First Attempt.}
The main task is therefore to show that such a decomposition into
blocks exists when~(\ref{eq_thm_hypothesis}) holds and $q$ is large
enough. A key concept in the construction of the blocks is the
notion of {\em good} vertices. Roughly speaking the blocks are
constructed in such a way that the boundary of each block consists
of good vertices only.

Good vertices $v$ are vertices that are of degree bounded
by $c$ and such that
\begin{equation} \label{eq:sum_leq_eps}
\sum_{u \neq v} \alpha^{d(u,v)} \leq \eps.
\end{equation}
A nice feature of this definition is that it is easy to see that
if all the vertices at a boundary of a block $V$
satisfy~(\ref{eq:sum_leq_eps}) then any vertex inside the block satisfies
the same inequality with $\alpha^2$ instead of $\alpha$.

Assume for a moment that all blocks constructed are trees.
In this case~(\ref{eq:sum_leq_eps})
implies that for a large enough $q$ and given
two boundary conditions that differ at one site, it is possible to couple
the configurations inside the block with expected hamming distance $\eps$.
Moreover, the case where all the blocks are trees, we show that
the second condition in~(\ref{eq_thm_hypothesis}) together with the small
effect of the boundary implies a polynomial relaxation time of the dynamics
inside the block.

\paragraph{Cyclic components and skeletons.}
More work is needed since we may not assume that all blocks are
trees. In fact, a crucial step of the construction is to show that
there are components $W_1,\ldots,W_r$ that contain all cycles of
length $O(\log n)$ and such that all degrees in $W_i$ are bounded,
the size of each $W_i$ is $O(\log n)$ and the distance between $W_i$
and $W_j$ is $\Omega(\log n)$. All of the properties above follow
from the assumption on the tree excess. We call the components $W_i$
the {\em skeletons}.

Given the skeletons $W_i$, we consider two types of blocks: tree
blocks and the blocks consisting of $W_i$ and trees attaching to
$W_i$. Using~(\ref{eq_thm_hypothesis}) we show that the mixing time
of each block is polynomial in $n$ and that the effect of the
boundary on each block is small. This allows to deduce a polynomial
mixing time bound.

\section{Proofs} \label{sec:proofs}

\subsection{Proof of Theorems~\ref{t:randomGraph} , \ref{t:randomGraphHard} and \ref{t:randomGraphSoft}}
\begin{proof}(Theorem \ref{t:randomGraph},\ref{t:randomGraphHard},\ref{t:randomGraphSoft})
The proofs follows by by Lemma \ref{l:randomGraph} below and
Theorems \ref{t:main}, \ref{t:mainHard} and \ref{t:mainSoft}
respectively.
\end{proof}

\begin{lemma}\label{l:randomGraph}
For every $d \geq 1$ there exist $0 < a,\alpha, t,\delta < \infty$ such
if $G$ is a random graph distributed according to $G(n,d/n)$ then with high
probability $m_{\alpha}(G,a \log n)\leq \delta \log n$ and for all $v \in
V$, $t(v,a \log n)\leq t$.
\end{lemma}

\begin{proof}
It is well known that $G(n,d/n)$ satisfies
$t(v,2 a \log n)\leq 1$ for all $v$ with high probability,
provided that $a = a(d) > 0$ is sufficiently small, see,
e.g.,~\cite{MosselSly:07}.
Next we show that if $\alpha$ is sufficiently small then with high probability
for all $v_0$ and all $\Gamma$, a
self-avoiding path of length $a \log n$ starting at the vertex $v_0$,
it holds that
\[
\sum(\Gamma) := \sum_{u \in \Gamma} \sum_{v : u \neq v \in G}
\alpha^{d(u,v)} \leq \delta\log n.
\]
Considering the contribution to the sum from $u \notin B(v,2 a \log n)$
we see that
\[
\sum(\Gamma) \leq
\sum_{u \in \Gamma} \sum_{v : u \neq v \in B(v_0,2 a \log n)} \alpha^{d(u,v)} +
(a \log n) \times n \times \alpha^{a \log n}.
\]
Note that $(a \log n) \times n \times \alpha^{a \log n} = o(1)$ if $\alpha > 0$
is small enough so that $a \log \alpha + 1 < 0$.
In order to bound the first sum we note that
\[
\sum_{u \in \Gamma} \sum_{v : u \neq v \in B(v_0,2 a \log n)} \alpha^{d(u,v)} \leq
\sum_{D=1}^{2 a \log n} \alpha^D \sum_{v \in B(v_0,2 a \log n)}
|\{ u \in \Gamma : d(v,u) = D \}|.
\]
Note that for each $v \in B(v_0,2 a \log n)$ the size of the set
$\{ u \in \Gamma : d(v,u) = D \}$ is at most $4$. Indeed suppose that there
are five elements $u_1,\ldots,u_5$ in this set. For $u_i$ denote by $u'_i$
the last point on $\Gamma$ on a shortest path from $u_i$ to $v$ and $w_i$ be
the following point. Since $\Gamma$ is a path it follows that the size of the set $\{ u'_i : 1 \leq i \leq 5\}$
is at least $3$. Without loss of generality assume that $u'_1,u'_2$ and
$u'_3$ are distinct. Then removing the edges $(u'_1,w_1)$ and $(u'_2,w_2)$
will maitain the connectivity properties of $B(v_0,2 a \log n)$ contradicting
the fact that $t(v_0,2 a \log n)\leq 1$. The argument above implies that
\[
\sum_{D=1}^{2 a \log n} \alpha^D \sum_{v \in B(v_0,2 a \log n)} |\{
u \in \Gamma : d(v,u) = D \}| \leq 4 \sum_{D=1}^{2 a \log n}
\alpha^D |\{ v \in B(v_0,2 a \log n) : d(v,\Gamma) \leq D\}|.
\]
We now use the well known expansion bounds implying that
 in $G(n,d/n)$ with high probability
all connected sets $\Gamma$ of size at least $a \log n$ have at most
$h^D |\Gamma|$ elements at distance at most $D$ from $\Gamma$ which allows
to bound the last sum as
\[
4 a \log n \sum_{D=1}^{2 a \log n} \alpha^D h^D \leq \frac{\delta}{2} \log n,
\]
provided $\alpha$ is small enough. Finally, we recall the proof of
the expansion bound. Note that it suffices to show that for all
connected sets $\Gamma$ of size at least $a \log n$, the number of
elements at distance exactly $1$ from the set is bounded by $(h-1)
|\Gamma|$. By a first moment calculation, the probability that a set
with more neighbors exists is bounded by:
\begin{align*}
&\quad\sum_{s=a \log n}^n \binom{n}{s} s! \left( \frac{d}{n}
\right)^{s-1} P[Bin(s (n-s),d/n) > (h-1) s]\\ &\leq \sum_{s=a \log
n}^n n d^{s-1} P[Bin(s n,d/n) > (h-1) s] = o(1),
\end{align*}
provided $h$ is large enough since by standard large deviation
results,
\begin{align*}
P[Bin(s n,d/n) > (h-1) s] &\leq E \exp(Bin(s n,d/n) - (h-1) s)\\
&=(1+\frac{d(e-1)}{n})^{sn}\exp(- (h-1) s) \\
&\leq \exp\left(s[d(e-1)-(h-1)])\right).
\end{align*}
\end{proof}

\subsection{Notation}

\begin{definition} \label{def:boundary}
Let $\partial U$ denote the {\em interior boundary} of $U$:
\[
\partial U = \{u\in U: \exists u' \in U^c \mbox{ s.t. }(u',u)\in E\}.
\]
Let $\partial^+ U$ denote the {\em exterior boundary} of $U$:
\[
\partial^+ U = \{u\in U^c: \exists u' \in U \mbox{ s.t. } (u',u)\in E\}
\]
For $U\subseteq W \subseteq V$ denote the
{\em exterior boundary of $W$ with respect to $U$}:
\[
\partial^+_W U = \{u\in W^c: \exists u' \in U \mbox{ s.t. } (u',u)\in E\}.
\]

If $T$ is a tree rooted at $\rho$ and $u\in T$ then we let $T_u$
denote the subtree of $u$ and all its descendants.  Let $T_u^+$
denote $T_u\cup \partial^+_T T_u$. 
\end{definition}




\begin{definition} \label{def:good_bad}
Define the {\em $\alpha$-weight} of a vertex $v$ by
$\varphi_\alpha(v)=\sum_{u\neq v} \alpha^{d(v,u)}$.  We call $v$ a
$(c,\alpha,\epsilon)$-{\em good} vertex if the degree of $v$ is less than
or equal to $c$ and $\varphi_\alpha(v)\leq\epsilon$.  If $v$ is not
a $(c,\alpha,\epsilon)$-good vertex then it is a
$(c,\alpha,\epsilon)$-{\em bad} vertex.  When there is no ambiguity in the
parameters $(c,\alpha,\epsilon)$ we will simply call vertices good
or bad vertices.
\end{definition}


\subsection{Relaxation and Mixing Times}
Although not necessary for our results, to make use of existing
theory it is convenient to make the assumption that the Gibbs
sampling is lazy, that is we introduce self-loop probability of a
half for all states. It is well known that Gibbs sampling is a
reversible Markov chain with stationary distribution $P$.  Let
$1=\lambda_1 > \lambda_2 \geq \ldots \geq \lambda_m \geq -1$ denote
the eigenvalues of the transition matrix of Gibbs sampling.  The
{\em spectral gap} is denoted by $\max\{1-\lambda_2,1-|\lambda_m|\}$
and the {\em relaxation time} $\tau$ is the inverse of the spectral
gap.  The relaxation time can be given in terms of the Dirichlet
form of the Markov chain by the equation
\begin{equation}\label{eq_relax_defn}
\tau=\sup\left\{\frac{2\sum_\sigma P(\sigma)
(f(\sigma))^2}{\sum_{\sigma\neq\tau} P(\sigma,\tau)
(f(\sigma)-f(\tau))^2} \right\}
\end{equation}
where $f$ is any function on configurations,
$P(\sigma,\tau)=P(\sigma)P(\sigma \rightarrow \tau)$ and $P(\sigma
\rightarrow \tau)$ is transition probability from $\sigma$ to
$\tau$. We use the result that the for reversible Markov chains the
relaxation time satisfies
\begin{equation} \label{eq:tau_and_spectral}
\tau \leq \tau_{mix}\leq  \tau\left(1+\frac12 \log(\min_\sigma
P(\sigma))^{-1}\right)
\end{equation}
where $\tau_{mix}$ is the mixing time (see e.g.
\cite{AldousFill:u}).  In all our examples we have $\log(\min_\sigma
P(\sigma))^{-1}=\hbox{poly}(n)$ so by bounding the relaxation time
we can bound the mixing time up to a polynomial factor.\\

For our proofs it will be useful to use the notion of {\em block
dynamics}. The Gibbs sampler can be generalized to update blocks of
vertices rather than individual vertices.  For blocks
$V_1,V_2,\ldots,V_k\subset V$, not necessarily disjoint, with
$V=\cup_i V_i$ the block dynamics of the Gibbs sampler updates a
configuration $\sigma$ by choosing a block $V_i$ uniformly at random
and assigning the spins in $V_i$ according to the Gibbs distribution
conditional on the spins on $G-\{V_i\}$.  The relaxation time of the
Gibbs sampler can be given in terms of the relaxation time of the
block dynamics and the relaxation times of the Gibbs sampler on the
blocks.
\begin{proposition} \label{prop:martinelli}
If $\tau_{block}$ is the relaxation time of the block dynamics and
$\tau_{i}$ is the maximum the relaxation time on $V_i$ given any
boundary condition from $G-\{V_i\}$ then by Proposition 3.4 of
\cite{Martinelli:99}
\begin{equation}\label{eq_relax_block}
\tau\leq \tau_{block}(\max_i \tau_{i})\max_{v\in V} \{\# j: v\in
V_j\}.
\end{equation}
\end{proposition}

\subsubsection{Canonical Paths and Conductance}
We will use the following conductance result which follows from
Cheeger's inequality, see e.g.,~\cite{JerrumSinclair:89}.
\begin{proposition} \label{prop:cheeger}
Consider an ergodic reversible
Markov chain $X_i$ on a discrete space $\Omega$ where
for any two states $a,b \in \Omega$ such that
$P(a,b) := P(a) P(a \to b) > 0$ it holds that
$P(a,b) > \eps$. Then
\[
\tau_{mix} \leq 2/\eps^2.
\]
\end{proposition}

\begin{proposition} \label{prop:canonical_path}
Suppose that for any two states $\sigma,\eta$ in the state space we
have a canonical path
$\gamma_{(\sigma,\eta)}=(\sigma=\sigma^{(0)},\sigma^{(1)},\ldots,\sigma^{(k)}=\eta)$
such that each transitions satisfies
$P(\sigma^{(i)},\sigma^{(i+1)})>0$. Let $L$ be the length of the
longest canonical path between two states and let
\[
\rho=\sup_{(\eta',\eta'')} \sum_{(\sigma,\eta):(\eta',\eta'')\in
\gamma_{(\sigma,\eta)}} \frac{P(\sigma)P(\eta)}{P(\eta',\eta'')}
\]
where the supremum is over pairs of states $\eta',\eta''$ with
$P(\eta',\eta'')>0$ while the sum is over all pairs of states. Then
the relaxation time satisfies
\[
\tau\leq L\rho.
\]
\end{proposition}



\subsubsection{Path Coupling}\label{s:pathCouple}
We use the path coupling technique~\cite{BubleyDyer:97} to bound the
relaxation time. The proposition below follows
from~\cite{BubleyDyer:97} and~\cite{Chen:98}, see
also~\cite{BeKeMoPe:05}. For two configurations $\sigma,\sigma' \in
\mathcal{C}^V$ we denote their {\em Hamming distance} by
$d_H(\sigma,\sigma') = |\{ v : \sigma(v) \neq \sigma'(v)\}|$.
\begin{proposition} \label{prop:path_coupling}
Consider Gibbs sampling on a graph $G$.
Suppose that for any pair of configurations $\sigma_1,\sigma_2$ that differ in one site
only, there is a way to couple the dynamics such that if $\sigma_1'$ and $\sigma_2'$ denote
the configuration after the update then:
\[
E[d_H(\sigma_1',\sigma_2')] \leq 1 - \frac{c}{n}.
\]
Then
\[
\tau_{mix} \leq c.
\]
\end{proposition}

\subsection{Block mixing}

For the proof we will consider block dynamics where the blocks are in some sense weakly
connected. We will bound the relaxation time of the block dynamics in terms of single site
dynamics of the sites connecting the blocks as follows.

\begin{lemma}\label{lemma_block_dynamics}
Let $P$ be any Gibbs measure taking values in $\mathcal{C}$. Let
$U\subset V$ and fix some boundary condition $\eta$ on $\partial^+
U$.
Suppose that $U$ is
the disjoint union of subsets $U_i$. Further suppose that for all $i$ there
exist $w_i \in U_i$ such that there are no edges between $U-U_i$ and
$U_i-\{w_i\}$. Let $W=\cup_i \{w_i\}$. Let $B_i=
\partial^+_U U_i$ and let
\begin{equation}\label{e:externalField}
p_{w_i}(x)=P_{U_i\cup B_i} (\sigma(w_i) =x |\sigma(B_i)=\eta(B_i)).
\end{equation}
We define the distribution $Q$ on $\mathcal{C}^W$ by
\begin{equation} \label{eq:defQ}
Q(\sigma(W))=\frac1{Z}\widehat{P}_W(\sigma(W))\prod_i
p_{w_i}(\sigma(w_i))
\end{equation}
where $\widehat{P}$ is the activity free distribution from
Definition \ref{def:models}. Then the relaxation time $\tau_Q$ of
Gibbs sampling for $Q$ satisfies $\tau_{block} \leq
\max(|W|,\tau_Q)$.



\end{lemma}

\begin{proof}
Let $P^{\eta}$ denote the probability measure on $U$ with boundary
conditions $\eta$. Then by the Markov property and (\ref{eq:defQ})
it follows that $P^{\eta}_W = Q$. We note furthermore that from the
Markov property it follows that the measure $P^{\eta}$ satisfies for
any $i$:
\begin{align} \label{eq:markov_complicated}
P^{\eta}(\sigma(B_i) = \sigma' | \sigma(U \setminus B_i) = \sigma'')
&= Q(\sigma(w_i) = \sigma'(w_i)   | \sigma(W \setminus \{w_i \}) =
\sigma''(W \setminus \{w_i\}) )\nonumber\\ & \quad \times
P^{\eta}(\sigma(B_i \setminus \{w_i\}) = \sigma'(B_i \setminus \{
w_i \}) |
 \sigma(w_i) = \sigma'(w_i)).
\end{align}
Write $\sigma_t$ for the state of the block dynamics with blocks
$B_i$ and boundary conditions $\eta$. Write $\sigma'_t$ for the
state of the single site dynamics for (\ref{eq:defQ}). Then assuming
that we have $\sigma_0(W) = \sigma'_0$ we obtain by equation
(\ref{eq:markov_complicated}) that the dynamics on $\sigma$ and
$\sigma'$ may be coupled in such a way that for all $t$:
\begin{itemize}
\item $\sigma_t(W) = \sigma'_t$.
\item If all the blocks (sites) in $\sigma_t$ ($\sigma'$) have been updated at
least once then:
\[
P(\sigma_t = \sigma^{\ast} | \sigma_t(W) = \sigma^{\ast \ast}) =
P^{\eta}(\sigma = \sigma^{\ast} | \sigma(W) = \sigma^{\ast \ast}).
\]
\end{itemize}
Note that the probability that at least one block has not been
updated by time $t$ is at most $|W| (1-1/|W|)^t$. Let $P^t$ denote
the distribution of $\sigma_t$ and similarly $Q^t$. Given an optimal
coupling between $Q^t$ and $Q$ consider the coupling of $P^t$ to $P$
where given two configurations $(\sigma'_1,\sigma'_2)$ distributed
according to the coupling, we let $\sigma_1$ be distributed
according to the conditional distribution given $\sigma_1'$ and
similarly for $\sigma_2$. Moreover by the argument above it follows
that we may define $\sigma_1$ and $\sigma_2$ is such a way that if
$\sigma'_1(W) = \sigma'_2(W)$ and all blocks have been updated then
$\sigma_1 = \sigma_2$. This implies that
\[
d_{TV}(P^t,P^{\eta}) \leq d_{TV}(Q^t,Q) + |W|(1-\frac{1}{W})^t.
\]
Since the relaxation time measures the exponential rate of
convergence to the distribution we conclude that $\tau_{block} \leq
\max(|W|,\tau_Q)$.
\end{proof}

Our bounds on the relaxations times of trees will be given in terms
of their path density defined as follow
\begin{definition}\label{d:pathDensity}
For a tree $T\subset G$ rooted at $\rho$ we let the {\em maximal path
density} be defined by
\[
m(T,\rho) = \max_{\Gamma} \sum_{u \in \Gamma} \hbox{deg}(u)
\]
where the maximum is taken over all self-avoiding paths $\Gamma
\subset T$ starting at $\rho$.
\end{definition}

\subsubsection{Colouring Model}

Next we prove two lemmas which will be used together with Lemma~\ref{lemma_block_dynamics}
to prove relaxation bounds below.
\begin{lemma}\label{lemma_star_mixing}
Let $W$ be a star with center $v$ and $k$ leaves.  Let
\[
Q(\sigma(W))=\frac1{Z}P_W(\sigma(W))\prod_{w\in W} p_{w}(\sigma(w))
\]
where the $p_{w}$ are functions such that for all $w\in W$,
$\sum_{x\in \mathcal{C}} p_w(x)=1$ and for all $w\in W,x\in
\mathcal{C}$ either $p_w(x)>(q\delta)^{-1}$ or $p_w(x)=0$. Further
assume that for some $c \leq q-3$ we have that for all $w\in W-{v}$,
$\#\{x\in \mathcal{C}:p_w(x)=0\}\leq c$. Then the relaxation time
$\tau$ of the Glauber dynamics on $Q$ is at most $C^{k}$ where $C$
is a constant depending only on $c,\delta,q$.
\end{lemma}

\begin{proof}
We first show that the chain is ergodic by constructing a path
between any two configurations $\sigma$ and $\eta$ with $Q(\sigma)$
and $Q(\eta)>0$. Since for each leaf $w$ there are at least $3$
colours $x$ with $p_w(x)>0$ we can find a colour $x(w)$ such that
$p_w(x(w))>0$ and $\sigma(v) \neq x(w) \neq \eta(v)$.  The path is
constructed by changing the states of the leaves to $x(u)$, then
changing the state of $v$ to $\eta(v)$, then finally changing the
states of the leaves to $\eta(u)$. Now by the hypothesis there are
at most $q^{k+1}$ colourings of $W$ so $Z\leq q^{k+1}$ so we have
that $Q(\sigma),Q(\eta)>(q^{2}\delta)^{-(k+1)}$. For two adjacent
states $\sigma$ and $\sigma'$ with $Q(\sigma),Q(\sigma')>0$, we have
$Q(\sigma\rightarrow\sigma')\geq (q\delta(k+1))^{-1}$ and so
$Q(\sigma,\sigma')\geq (q^{2}\delta)^{-(k+1)}(q\delta(k+1))^{-1}$.
From Proposition~\ref{prop:cheeger} it now follows that
\[
\tau_2 \leq ((q \delta (k+1))^2 (q^2 \delta)^{k+1})^4 \leq
4^k q^{20k} \delta^{20k},
\]
as needed.
\end{proof}

Similarly, it is easy to see that
\begin{lemma}\label{lemma_small_graph_mixing}
Let $W$ be a graph with $k$ vertices of maximum degree $d$.  Let
\[
Q(\sigma(W))=\frac1{Z}P_W(\sigma(W))\prod_{w\in W}
p_{w}(\sigma(w_i))
\]
where the $p_{w}$ are functions such that for all $w\in W$,
$\sum_{x\in \mathcal{C}} p_w(x)=1$ and for all $w\in W,x\in
\mathcal{C}$ either $p_w(x)>(q\delta)^{-1}$ or $p_w(x)=0$. Further,
for some $c \leq q-d-2$ we have that for all $w\in W$, $\#\{x\in
\mathcal{C}:p_w(x)=0\}\leq c$.  Then the relaxation time of the
Glauber dynamics on $Q$ is at most $C^{k}$ where $C$ is a constant
depending only on $c,\delta,d$ and $q$.
\end{lemma}

We can now obtain polynomial mixing time results for the type of blocks that
will be used in the construction.
\begin{theorem}\label{theorem_tree_mixing}
Let $T\subseteq U\subset V$ such that $T$ is a tree rooted at $\rho$
and so that there are no edges between $T-\{\rho\}$ and $U-T$.
Suppose that for all $u\in T$, $\#\{v\in V-U:(v,u)\in E\}< c$ and
that for each $u\in T$,
\begin{equation} \label{eq:corr_decay_tree}
\sup_{\sigma(\partial^+_U T_u)} \sup_{x\in \mathcal{C}}
\sup_{y\in\mathcal{C} :P_{T_u^+}(\sigma(u)=y|\sigma(\partial^+_U
T_u))\neq 0} \frac{P_{T_u^+}(\sigma(u)=x|\sigma(\partial^+
T))}{P_{T_u^+}(\sigma(u)=y|\sigma(\partial^+ T))} \leq \delta
\end{equation}

For some $l\geq 1$ assume there are at most $l$ edges between
$\{\rho\}$ and $U-T$. Let $\tau$ be the relaxation time of the
Glauber dynamics on $T$. If $q\geq c+l+2$ then for any boundary
condition $\eta$ on $\partial^+ T$ we have that $\tau \leq
C^{m(T,\rho)}$ where $m(T,\rho)$ is the maximal path density on $T$
and where $C$ is a constant depending only on $c,\delta,q$ and $l$.

\end{theorem}

\begin{proof}
We proceed by induction on $m(T,\rho)$.  If $T$ is a single point
then $\tau=1$ and so $\tau \leq C^{m(T,\rho)}$.  Now suppose $\rho$
has children $u_1,\ldots,u_k\in T$.  By induction the relaxation
time of the Glauber dynamics on $T_{u_i}$, $\tau_i \leq
C^{m(T_{u_i},u_i)}$ and by the definition of the maximal path
density $m(T_{u_i},u_i) \leq m(T,\rho) - k$.  Let $\tau_{block}$
denote the block dynamics on $T$ with blocks
$\{\{\rho\},T_{u_1},\ldots, T_{u_k}\}$.  Applying Lemma
\ref{lemma_block_dynamics} and \ref{lemma_star_mixing} we get that
the block dynamics satisfies $\tau_{block} \leq C^{k}$.  Then by
Proposition 3.4 of \cite{Martinelli:99} we have that
\[
\tau \leq \tau_{block} \max_i \{1, \tau_i\}\leq C^{k} C^{m(T,\rho) -
k} \leq  C^{m(T,\rho)}
\]
which completes the result.
\end{proof}

\subsubsection{Hardcore Model}

\begin{lemma}\label{lemma_hard_small_mixing}
Let $W$ be a graph and let
\[
Q(\sigma(W))=\frac1{Z}\widehat{P}_W(\sigma(W))\prod_{w\in W}
p_{w}(\sigma(w_i))
\]
where the $p_{w}$ are functions such that for some $\delta$ and all
$w\in W$, $\delta<p_w(0)<1$ and $p_w(0) + p_w(1) = 1$. Then the
relaxation time $\tau$ of the Glauber dynamics of $Q$ satisfies
$\tau\leq C^{|W|}$ where $C$ depends only on $\beta$ and $\delta$.
\end{lemma}

\begin{proof}
We use the method of canonical paths from Proposition
\ref{prop:canonical_path}. Let $\sigma$ and $\eta$ be two
configurations with $Q(\sigma)$ and $Q(\eta)>0$. We define the
canonical path to be a path which begins from $\sigma$, then
sequentially changes states of all the vertices to 0 and then
sequentially changes the state of $w\in W$ to 1 if $\eta(w)=1$. Now
suppose $\eta',\eta''$ is a step in some path.  Clearly each path is
of length at most $2|W|$. They must differ at exactly one site $w\in
W$ and suppose that $\eta'(w)=1$ and $\eta''(w)=0$. If
$(\eta',\eta'')$ is in the canonical path $\gamma_{(\sigma,\eta)}$
then $\sigma\geq \eta'$ under the canonical partial ordering. Now
$P[\eta'\rightarrow\eta'']=\frac{p_w(0)}{|W|}\geq
\frac{\delta}{|W|}$.  Then
\begin{align*}
\sum_{(\sigma,\eta):(\eta',\eta'')\in \gamma_{(\sigma,\eta)}}
\frac{P(\sigma)P(\eta)}{P(\eta',\eta'')}&\leq
\sum_{\sigma:\sigma\geq \eta'} \frac{P(\sigma)}{P(\eta',\eta'')}\\
&= P[\eta'\rightarrow\eta'']^{-1}\sum_{\sigma:\sigma\geq \eta'}
\frac{\exp(\beta\sum_u \sigma(u))\prod_u p_w(\sigma(u))}
{\exp(\beta\sum_u \eta'(u))\prod_u p_w(\eta'(u))}\\
&\leq \frac{|W|}{\delta}((1+\exp(\max(\beta,0))\delta^{-1})^{|W|}.
\end{align*}
Similarly the same bound holds for pairs with $\eta'(w)=0$ and
$\eta''(w)=1$ so $\rho\leq
\frac{|W|}{\delta}((1+\exp(\max(\beta,0))\delta^{-1})^{|W|}$. From
Proposition~\ref{prop:canonical_path} it now follows that
\[
\tau_2 \leq
\frac{2|W|^2}{\delta}((1+\exp(\max(\beta,0))\delta^{-1})^{|W|} \leq
10^{|W|} \exp(\max(\beta,0)|W|) \delta^{-|W|},
\]
as needed.
\end{proof}

\begin{theorem}\label{theorem_tree_mixing_hardcore}
Let $T\subset V$ be a tree rooted at $\rho$. Then $\tau \leq
C^{m(T,\rho)}$ where $m(T,\rho)$ is the maximal path density on $T$
and where $C$ is a constant depending only on $\beta$.
\end{theorem}

\begin{proof}
We proceed by induction on $m(T,\rho)$.  If $T$ is a single point
then $\tau=1$ and so $\tau \leq C^{m(T,\rho)}$.  Now suppose $\rho$
has children $u_1,\ldots,u_k\in T$.  By induction the relaxation
time of the Glauber dynamics on $T_{u_i}$ satisfies $\tau_i \leq
C^{m(T_{u_i},u_i)}$.  By definition of the maximal path density
$m(T_{u_i},u_i) \leq m(T,\rho) - k$.  Let $\tau_{block}$ denote the
block dynamics on $T$ with blocks $\{\{\rho\},T_{u_1},\ldots,
T_{u_k}\}$. We define the distribution $Q$ on $\mathcal{C}^W$ by
\[
Q(\sigma(W))=\frac1{Z}\widehat{P}_W(\sigma(W))\prod_{w\in W}
p_{w_i}(\sigma(w_i))
\]
and $p_{w_i}$ is as in equation~\eqref{e:externalField}.  Applying
Lemma \ref{lemma_block_dynamics} with $W=\{\rho,u_1,\ldots,u_k\}$
implies that $\tau_{block} \leq \max(k+1,\tau_Q)$ where $\tau_Q$ is
the relaxation time of the Glauber dynamics on the measure $Q$.  In
the hardcore model for any vertex $v$ and any boundary condition
$\sigma(V-\{v\})$ on $V-\{v\}$ we have that
$P(\sigma(v)=0|\sigma(V-\{v\}))\geq \frac1{1+e^\beta}$, the
probability that the spin at $v$ is 0 given that the spins of all
its neighbors are 0, and so each $p_{w}(0)\geq\frac1{1+e^\beta}$. It
follows that in Lemma~\ref{lemma_hard_small_mixing} we can take
$\delta=\frac1{1+e^\beta}$ and so $\tau_{block} \leq
\max(k+1,C_1^{k+1})\leq C^k$ for sufficiently large $C$. Then by
Proposition 3.4 of \cite{Martinelli:99} we have that
\[
\tau \leq \tau_{block} \max_i \{1, \tau_i\}\leq C^{k} C^{m(T,\rho) -
k} \leq  C^{m(T,\rho)}
\]
which completes the result.
\end{proof}

\subsubsection{Soft constraint Models}

For soft constraint models, bounding the mixing time is
simplified by the fact that removing an edge adds at most a constant
multiplicative factor to the relaxation time.

\begin{theorem}\label{thm_relax_softcore}
Let $\tau$ be the relaxation time  of the Glauber dynamics on a tree
$T\subset V$. Given arbitrary boundary conditions,
\[
\tau \leq \exp(4 \|H\| m(T))
\]
where $\|H\|$ is the norm of the Hamiltonian.
\end{theorem}

\begin{proof}

We proceed by induction on $m$ with a similar argument to the one
used in \cite{MosselSly:07} for the Ising model. Note that if $m=0$
the claim holds true since $\tau = 1$. For the general case, let $v$
be the root of $T$, and denote its children by $u_1, \ldots,u_k$ and
denote the subtree of the descendants of $u_i$ by $T^i$.  Now let
$T'$ be the tree obtained by removing the $k$ edges from $v$ to the
$u_i$, let $P'$ be the model on $T'$ and let $\tau'$ be the
relaxation time on $T' $. By equation \eqref{eq_relax_defn} we have
that
\begin{equation} \label{eq:recursion_ising}
\tau /\tau' \leq \frac{\max_\sigma
P(\sigma)/P'(\sigma)}{\min_{\sigma,\tau}P(\sigma,\tau)/P'(\sigma,\tau)}
\leq \exp(4 \|H\| k).
\end{equation}
Now we divide $T'$ into $k+1$ blocks $\{\{v\},T^1,\ldots,T^k \}$.
Since these blocks are not connected to each other the mixing time
of the block dynamics is simply $1$. By applying Proposition 3.4 of
\cite{Martinelli:99} we get that the relaxation time on $T'$ is
simply the maximum of the relaxation times on the blocks,
\[
\tau' \leq \max \{1,\tau^i\}.
\]
where $\tau^i$ is the relaxation time on $T^i$. Note that by the
definition of $m$, it follows that the value of $m$ for each of the
subtrees $T^i$ satisfies $m(T^i) \leq m - k$, and therefore for all
$i$ it holds that $\tau^i \leq \exp(4 \|H\| (m-k))$. This then
implies by~(\ref{eq:recursion_ising}) that $\tau \leq \exp(4 \|H\|
m)$ as needed.
\end{proof}

\subsection{Correlation Decay  in Tree Blocks}
In this subsection we prove that if we look at a tree block, all of whose
leaves are good, then for large enough $q$ we have the correlation decay
property~(\ref{eq:corr_decay_tree}).

\begin{definition}
For $0<\lambda<1$ and $U\subset V$ define the block boundary
weighting as the function defined by:
\[
\psi_\lambda(v)=\psi(v)= \sum_{w\in \partial^+ U} \lambda^{d(w,v)},
\]
for all $v \in U$.
\end{definition}

\begin{lemma}\label{l:treeboundaryweight}
If every vertex in  $\partial^+ U$ is $(c,\alpha,\epsilon)$-good
then for all $\lambda\leq \alpha^2$,
\[
\psi(v)\leq \frac{\epsilon\lambda}{\alpha^2}
\]
\end{lemma}

\begin{proof}
Let $v\in U$ and let $u\in \partial^+ U$ be an exterior boundary
vertex which minimizes the distance to $v$. Then
\begin{equation}\label{e:psi}
\psi_{\alpha^2}(v)  \leq \sum_{w\in
\partial^+ U} \alpha^{(d(v,u)+d(u,w))} \leq \sum_{w\neq u}
\alpha^{d(w,u)}=\varphi_\alpha(u) \leq \epsilon.
\end{equation}
and the result follows since for $\lambda\leq \alpha^2$ we have
$\psi_\lambda(v)\leq \frac{\lambda}{\alpha^2}\psi_{\alpha^2}(v)$.
\end{proof}

\subsubsection{Colouring}

\begin{lemma}\label{l:treeCorrDecay}
Suppose that $T = (V_T,E_T)$ is an induced subgraph of $G=(V,E)$
that is a tree and suppose that for all $v\in V_T$, $\psi(v)\leq 1$.
Then there exists a $q$ depending only on $\lambda$ such that for
all $v \in V_T$:
\begin{equation} \label{eq:sup_ratio_tree}
\sup_{\sigma(\partial^+ T)} \sup_{x\in \mathcal{C}}
\sup_{y\in\mathcal{C} :P(\sigma(v)=y|\sigma(\partial^+ T))\neq 0}
\frac{P(\sigma(v)=x|\sigma(\partial^+
T))}{P(\sigma(v)=y|\sigma(\partial^+ T))} \leq \exp(\psi(v))
\end{equation}
where the supremum is over all boundary conditions
$\sigma(\partial^+ T)$ on $\partial^+ T$.
\end{lemma}

\begin{proof}

Fix $v$ as the root of the tree. We will prove the result by
induction on the size of the tree. When the tree consists of a
single vertex $v$ the quantity in the left hand side
of~(\ref{eq:sup_ratio_tree}) is clearly $1$.

Let $u_1,\ldots,u_l$ be the children of $v$ in $T$. Consider the
graph $G'=(V',E')$ obtained from $G$ by removing the vertex $v$ and
all adjacent edges. Let
\begin{equation} \label{eq:def_deltai}
\delta_i= \sup_{\sigma(\partial^+_T T_{u_i})} \sup_{x\in
\mathcal{C}} \sup_{y\in\mathcal{C}
:P_{T^+_{u_i}}(\sigma(u_i)=y|\sigma(\partial^+_T T_{u_i}))\neq 0}
\frac{P_{T^+_{u_i}}(\sigma(u_i)=x|\sigma(\partial^+_T
T_{u_i}))}{P_{T^+_{u_i}}(\sigma(u_i)=y|\sigma(\partial^+_T
T_{u_i}))}
\end{equation}
For $w' \in T_{u_i}$ write $\psi_i(w')=\sum_{w\in \partial^+_T
T_{u_i}} \lambda^{d(w,w')}$. Note that $\psi_i$ is the function
$\psi$ for the subtree $T_{u_i}$ in the graph $G'$. Note moreover
that for all $w$ we have $\psi_i(w) \leq \psi(w)$. By the induction
hypothesis we therefore have $\delta_i \leq \exp(\psi_i(u_i))$. Let
$d_i=\#\{ w\in V' \setminus T_{u_i} : (w,u_i)\in E\}$ and note that
there are at least $q-d_i$ elements $y\in \mathcal{C}$ with
$P_{T^+_{u_i}}(\sigma(v)=y|\sigma(\partial^+ T_{u_i}))>0$ so
\[
\min_y \{P_{T^+_{u_i}}(\sigma(v)=y|\sigma(\partial^+
T_{u_i})):P_{T^+_{u_i}}(\sigma(v)=y|\sigma(\partial^+ T_{u_i})) >
0\}
 \leq \frac1{q-d_i}
\]
 and so by~(\ref{eq:def_deltai}) we have
\begin{equation} \label{eq:max_Ti}
\max_{y} P_{T^+_{u_i}}(\sigma(v)=y|\sigma(\partial^+ T_{u_i})) \leq
\frac{\delta_i}{q-d_i}.
\end{equation}
Since $d_i \lambda \leq \psi_i(u_i) \leq 1$,
taking $q > 2/\lam$ yields $q-d_i > q/2$.
When $0\leq x\leq 1$ we have $e^x-1
\leq 2x$ so $\delta_i-1\leq 2\psi(x)$. And since $\frac{x}{1-x}$ is
increasing in $x$
\begin{align*}
\sup \frac{1-P_{T^+_{u_i}}(\sigma(v)=x|\sigma(\partial^+
T_{u_i}))}{1-P_{T^+_{u_i}}(\sigma(v)=y|\sigma(\partial^+ T_{u_i}))}
&= 1 + \sup  \frac{P_{T^+_{u_i}}(\sigma(v)=y|\sigma(\partial^+
T_{u_i}))-P_{T^+_{u_i}}(\sigma(v)=x|\sigma(\partial^+ T_{u_i}))}{1-P_{T^+_{u_i}}(\sigma(v)=y|\sigma(\partial^+ T_{u_i}))}\\
&\leq 1 + \frac{\frac{\delta_i-1_{\{d_i=0\}}}
{q-d_i}}{1-\frac{\delta_i}{q-d_i}}
\mbox{ (By~(\ref{eq:max_Ti}) and since $\frac{x}{1-x}$ is increasing) }\\
&= 1 + \frac{\delta_i-1_{\{d_i = 0\}}}{q-d_i-\delta_i} \\
& \leq 1 + \frac{\delta_i -1_{\{d_i=0\}}}{q/2-e} \mbox{ (since $\delta_i \leq e$ and $q-d_i > q/2$) } \\
& \leq 1 + \frac{4 (\delta_i -1_{\{d_i=0\}})}{q}  \mbox{ (taking $q \geq 4e$) } \\
&\leq 1+ \frac{8\psi_i(u_i) + 4d_i}{q} \mbox{ (since $\delta_i-1\leq 2\psi(x)$) } \\
&\leq \exp(\frac{8\psi_i(u_i) + 4d_i}{q})\\
\end{align*}
where the supremum is taken over all $x,y\in \mathcal{C}$ and
boundary conditions on $\partial^+ T_u$.  Now note $\psi(v) \geq
\lambda\sum_i \psi_i(u_i)$ (it may be strictly greater due to the
contribution of the neighbors of $v$ outside $T$). Therefore:
\begin{align*}
\sup_{\sigma(\partial^+ T)} \sup_{x\in \mathcal{C}}
\sup_{y\in\mathcal{C} :P(\sigma(v)=y|\sigma(\partial^+ T))\neq 0}
\frac{P(\sigma(v)=x|\sigma(\partial^+
T))}{P(\sigma(v)=y|\sigma(\partial^+ T))} &=  \prod_i \sup
\frac{1-P_{T_{u_i}}(\sigma(v)=x|\sigma(\partial^+ T_{u_i}))}{1-P_{T_{u_i}}(\sigma(v)=y|\sigma(\partial^+ T_{u_i}))}\\
&\leq \exp(\frac{8\psi_i(u_i) + 4d_i}{q})\\
&\leq \exp([\frac{8}{q\lambda}+\frac{4}{q\lambda^2}]\psi(v))
\end{align*}
which completes the induction provided that $q$ is large enough so
that $q \geq \max(4 e,\frac{8}{\lambda}+\frac{4}{\lambda^2})$.
\end{proof}

The following corollary follows immediately from
Lemma~\ref{l:treeCorrDecay} and Lemma~\ref{l:treeboundaryweight}.
\begin{corollary}\label{c:treeCorrDecay}
For all $c, \alpha > 0$ and $\eps > 0$ there exists a $q$ for which
the following holds.
Let $T\subset V$ be a tree such that every vertex in $\partial^+ T$
is $(c,\alpha,\epsilon)$-good.  Then for any $0<\lambda<1$ there
exists a $q$ such that
\[
\sup_{\sigma(\partial^+ T)} \sup_{x\in \mathcal{C}}
\sup_{y\in\mathcal{C} :P(\sigma(v)=y|\sigma(\partial^+ T))\neq 0}
\frac{P(\sigma(v)=x|\sigma(\partial^+
T))}{P(\sigma(v)=y|\sigma(\partial^+ T))} \leq \exp(\sum_{w\in
\partial^+ T} \lambda^{d(w,v)})
\]
where the supremum is over all boundary conditions
$\sigma(\partial^+ U)$ on $\partial^+ U$.
\end{corollary}

\subsubsection{Hardcore model}

\begin{lemma}\label{l:treeCorrDecayHard}
Suppose that $T = (V_T,E_T)$ is an induced subgraph of $G=(V,E)$
that is a tree. For $v\in V_T$ and $\eta$ a boundary condition on
$\partial^+ T$ let $P^{\eta}$ denote the  measure
$P(\sigma(v)=\cdot|\sigma(\partial^+ U))$. Then if
$\beta_\lambda=\log \lambda$ then for all $\beta<\beta_\lambda$ and
$v \in V_T$:
\begin{equation} \label{eq:hardcoreTreeDecay}
d_{TV}(P^{\eta^1},P^{\eta^2})\leq \psi_\lambda(v)
\end{equation}
for any two boundary conditions $\eta^1$ and $\eta^2$ on $\partial^+
T$ where $d_{TV}$ is the total variation distance.
\end{lemma}

\begin{proof}
Since the left hand side of equation~\eqref{eq:hardcoreTreeDecay} is
bounded by 1 we can assume that $\psi(v)\leq 1$. Fix $v$ as the root
of the tree. We will prove the result by induction on the size of
the tree.  Let $u_1,\ldots,u_l$ be the children of $v$ in $U$ and
let $w_1,\ldots,w_m$ be the children of $v$ in $\partial^+ T$.
Consider the graph $G'=(V',E')$ obtained from $G$ by removing the
vertex $v$ and all adjacent edges and let $P^{\eta}_{T_{u_i}}$
denote $P'(\sigma(u_i)=\cdot|\eta)$.  Then
\begin{align}\label{eq:hardTVDist}
d_{TV}(P^{\eta^1},P^{\eta^2})
&= \left|P(\sigma(v)=0|\eta^1)-P(\sigma(v)=0|\eta^2)\right|\nonumber\\
&= \left| \frac1{1+e^{\beta}\prod_{i=1}^l
P^{\eta^1}_{T_{u_i}}(0)\prod_{i=1}^m
1_{\{\eta^1_{w_i=0}\}}}-\frac1{1+e^{\beta}\prod_{i=1}^l
P^{\eta^2}_{T_{u_i}}(0)\prod_{i=1}^m 1_{\{\eta^2_{w_i=0}\}}}\right|\nonumber\\
&\leq e^{\beta}\left| \prod_{i=1}^l
P^{\eta^1}_{T_{u_i}}(0)\prod_{i=1}^m 1_{\{\eta^1_{w_i=0}\}} -
\prod_{i=1}^l P^{\eta^2}_{T_{u_i}}(0)\prod_{i=1}^m
1_{\{\eta^2_{w_i=0}\}} \right|\nonumber\\
&\leq \begin{cases} \lambda & m\geq 1\\
e^{\beta}\left| \prod_{i=1}^l P^{\eta^1}_{T_{u_i}}(0) -
\prod_{i=1}^l P^{\eta^2}_{T_{u_i}}(0)\right| & m=0
\end{cases}
\end{align}
Now if $m\geq 1$ then $\psi(v)\geq \lambda$ so
$d_{TV}(P^{\eta^1},P^{\eta^2}) \leq \psi(v)$.  This establishes
equation \eqref{eq:hardcoreTreeDecay} for trees of size 1.  We now
proceed by induction.

Observe the simple inequality that if $0\leq x_1,\ldots,x_q \leq 1$
and $0\leq y_1,\ldots,y_q\leq 1$ then
\begin{align}\label{eq_simple_inequality}
\left | \prod_{l=1}^q x_l - \prod_{l=1}^q y_l \right | &= \left |
\sum_{j=1}^q (x_j-y_j)\prod_{l=1}^{j-1} x_l \prod_{l=j+1}^q
y_l \right |\nonumber\\
&\leq \sum_{j=1}^q \left | x_j-y_j\right |.
\end{align}
Applying equation \eqref{eq_simple_inequality} to
equation~\eqref{eq:hardTVDist} we get that when $m=0$,
\[
d_{TV}(P^{\eta^1},P^{\eta^2}) \leq e^{\beta} \sum_{i=1}^l
|P^{\eta^1}_{T_{u_i}}(0) - P^{\eta^2}_{T_{u_i}}(0)|.
\]
By the inductive hypothesis applied to the tree $T_{u_i}$ we have
that
\[
|P^{\eta^1}_{T_{u_i}}(0) - P^{\eta^2}_{T_{u_i}}(0)| \leq \sum_{w\in
\partial^+ T_{u_i}} \lambda^{d(w,u_i)}=\frac1{\lambda} \sum_{w\in
\partial^+ T_{u_i}} \lambda^{d(w,v)}
\]
so
\[
d_{TV}(P^{\eta^1},P^{\eta^2}) \leq e^{\beta} \sum_{i=1}^l
|P^{\eta^1}_{T_{u_i}}(0) - P^{\eta^2}_{T_{u_i}}(0)|\leq \psi(v)
\]
which completes the induction.
\end{proof}

\subsubsection{Soft constraint models}

\begin{lemma}\label{l:treeCorrDecayHard}
Suppose that $T = (V_T,E_T)$ is an induced subgraph of $G=(V,E)$
that is a tree. For $v\in V_T$ and $\eta$ a boundary condition on
$\partial^+ T$ let $P^{\eta}$ denote the the measure
$P(\sigma(v)=\cdot|\sigma(\partial^+ U))$. Then there exists an
$H^\lambda > 0$ depending only on $\lambda$ such that if $\| H
\|<H^\lambda$ and $v \in V_T$:
\begin{equation} \label{eq:softcoreTreeDecay}
d_{TV}(P^{\eta^1},P^{\eta^2})\leq \psi_\lambda(v)
\end{equation}
for any two boundary conditions $\eta^1$ and $\eta^2$ on $\partial^+
T$ where $d_{TV}$ is the total variation distance.
\end{lemma}

\begin{proof}
Since the left hand side of equation~\eqref{eq:softcoreTreeDecay} is
bounded by 1 we can assume that $\psi(v)\leq 1$.  Let $K=4
(e^{\|H\|}-e^{-\|H\|})$. We can take $H^\lambda$ to be small enough
so that $4K<\lambda$ and for $0\leq x \leq 1/\lambda$ we have
$\exp(-xK)\leq 1-xK/2$ and $\exp(2Kx)\leq 1+4Kx$. Fix $v$ as the
root of the tree. We will prove the result by induction on the size
of the tree. Let $u_1,\ldots,u_l$ be the children of $v$ in $U$ and
let $u_{l+1},\ldots,u_m$ be the children of $v$ in $\partial^+ T$.
Consider the graph $G'=(V',E')$ obtained from $G$ by removing the
vertex $v$ and all adjacent edges, let $P'$ denote the induced soft
constraint model on $G'$ and let $P^{\eta}_{T_{u_i}}$ denote
$P'(\sigma(u_i)=\cdot|\eta)$. Then for all $i$ and $z\in
\mathcal{C}$,
\begin{align*}
\frac{\sum_{y_i\in
\mathcal{C}}e^{g(z,y_i)}P^{\eta^1}_{T_{u_i}}(y_i)} {\sum_{y_i\in
\mathcal{C}}e^{g(z,y_i)}P^{\eta^2}_{T_{u_i}}(y_i)} &= 1 -
\frac{\sum_{y_i\in
\mathcal{C}}e^{g(z,y_i)}(P^{\eta^2}_{T_{u_i}}(y_i)-P^{\eta^1}_{T_{u_i}}(y_i))}
{\sum_{y_i\in \mathcal{C}}e^{g(z,y_i)}P^{\eta^2}_{T_{u_i}}(y_i)}\\
&\geq 1 - 2(e^{\|H\|}-e^{-\|H\|})
d_{TV}(P^{\eta^1}_{T_{u_i}},P^{\eta^2}_{T_{u_i}})\\
&\geq \exp(-K d_{TV}(P^{\eta^1}_{T_{u_i}},P^{\eta^2}_{T_{u_i}}))
\end{align*}
Similarly we have
\[
\frac{\sum_{y_i\in
\mathcal{C}}e^{g(z,y_i)}P^{\eta^1}_{T_{u_i}}(y_i)} {\sum_{y_i\in
\mathcal{C}}e^{g(z,y_i)}P^{\eta^2}_{T_{u_i}}(y_i)}\leq \exp(K
d_{TV}(P^{\eta^1}_{T_{u_i}},P^{\eta^2}_{T_{u_i}}))
\]
Then for each $x\in \mathcal{C}$,
\begin{align*}
\frac{P^{\eta^1}(v)(x)}{P^{\eta^2}(v)(x)} &=
\frac{e^{h(x)}\prod_{i=1}^m \sum_{y_i\in
\mathcal{C}}e^{g(x,y_i)}P^{\eta^1}_{T_{u_i}}(y_i)}{\sum_{z\in
\mathcal{C}}e^{h(z)}\prod_{i=1}^m \sum_{y_i\in
\mathcal{C}}e^{g(z,y_i)}P^{\eta^1}_{T_{u_i}}(y_i)}/\frac{e^{h(x)}\prod_{i=1}^m
\sum_{y_i\in
\mathcal{C}}e^{g(x,y_i)}P^{\eta^2}_{T_{u_i}}(y_i)}{\sum_{z\in
\mathcal{C}}e^{h(z)}\prod_{i=1}^m \sum_{y_i\in
\mathcal{C}}e^{g(z,y_i)}P^{\eta^2}_{T_{u_i}}(y_i)}\\
&= \frac{e^{h(x)}\prod_{i=1}^m \sum_{y_i\in
\mathcal{C}}e^{g(x,y_i)}P^{\eta^1}_{T_{u_i}}(y_i)}{e^{h(x)}\prod_{i=1}^m
\sum_{y_i\in
\mathcal{C}}e^{g(x,y_i)}P^{\eta^2}_{T_{u_i}}(y_i)}/\frac{\sum_{z\in
\mathcal{C}}e^{h(z)}\prod_{i=1}^m \sum_{y_i\in
\mathcal{C}}e^{g(z,y_i)}P^{\eta^1}_{T_{u_i}}(y_i)}{\sum_{z\in
\mathcal{C}}e^{h(z)}\prod_{i=1}^m \sum_{y_i\in
\mathcal{C}}e^{g(z,y_i)}P^{\eta^2}_{T_{u_i}}(y_i)}\\
&\leq \exp\left(2K \sum_{i=1}^m
d_{TV}(P_{T_{u_i}}^{\eta^1},P_{T_{u_i}}^{\eta^2})\right).
\end{align*}
Then
\begin{align*}
d_{TV}(P^{\eta^1},P^{\eta^2}) &= \sum_{x \in
\mathcal{C}} |P^{\eta^1}(x) - P^{\eta^2}(x)| \\
&= \sum_{x \in
\mathcal{C}} P^{\eta^2}(x)\left|\frac{P^{\eta^1}(x)}{P^{\eta^2}(x)}-1\right| \\
&\leq \exp\left(2K \sum_{i=1}^m
d_{TV}(P_{T_{u_i}}^{\eta^1},P_{T_{u_i}}^{\eta^2})\right) -1
\end{align*}
Now suppose that $T$ is a single vertex $\{v\}$ so $u_1,\ldots u_m$
are all in $\partial^+ T$ and so $\psi(v)=m\lambda$.  If $m=0$ then
$d_{TV}(P^{\eta^1},P^{\eta^2})=\psi(v)=0$.  If $1\leq m \leq
1/\lambda$ then
\[
d_{TV}(P^{\eta^1},P^{\eta^2}) \leq \exp\left(2K m)\right) -1 \leq 4K
m \leq \lambda m = \psi(v)
\]
while if $m>1/\lambda$ then $\psi(v)>1$.  So this verifies the case
when $T$ is a single point.  For the induction step our inductive
hypothesis says that
\[
d_{TV}(P^{\eta^1}_{T_{u_i}},P^{\eta^2}_{T_{u_i}})\leq \sum_{w\in
\partial^+ T_{u_i}} \lambda^{d(w,u_i)}=\frac1{\lambda} \sum_{w\in
\partial^+ T_{u_i}} \lambda^{d(w,v)}.
\]
If $\psi(v)\leq 1$ then $\sum_{i=1}^m
d_{TV}(P_{T_{u_i}}^{\eta^1},P_{T_{u_i}}^{\eta^2})\leq
\frac1{\lambda}$ and so
\[
d_{TV}(P^{\eta^1},P^{\eta^2}) \leq \exp\left(2K \sum_{i=1}^m
d_{TV}(P_{T_{u_i}}^{\eta^1},P_{T_{u_i}}^{\eta^2})\right) -1 \leq 4K
d_{TV}(P_{T_{u_i}}^{\eta^1},P_{T_{u_i}}^{\eta^2}) \leq \psi(v)
\]
which completes the induction.
\end{proof}

\subsection{Block Construction}

\begin{lemma}\label{l:equivalence_relation}
For two $(c,\alpha,\epsilon)$-bad points $u,u'$ we define $u\sim u'$ if
there is a path $u=u_1,u_2,\ldots,u_k=u'$ such that no two
consecutive vertices on the path $u_i,u_{i+1}$ are
$(c,\alpha,\epsilon)$-good. Then $\sim$ is an equivalence relation of
$(c,\alpha,\epsilon)$-bad vertices in $G$.
\end{lemma}

\begin{proof}
The relation is clearly reflexive and symmetric. Suppose that there
is a path $u \sim u'$ and $u \sim u''$.  Then there exist paths
$u=v_1,v_2,\ldots,v_k=u'$ and $u=w_1,w_2,\ldots,w_l=u''$ such that
no two consecutive vertices are $(c,\alpha,\epsilon)$-good.  Let
$i=\max(j:v_j\in \{w_1,w_2,\ldots,w_l\})$ and suppose that
$v_i=w_j$. Then the path
$u'=v_k,v_{k-1},\ldots,v_i,w_{j+1},w_{j+2},\ldots,w_l=u''$ is a path
with no two consecutive $(c,\alpha,\epsilon)$-good vertices so
$u'\sim u''$.  Hence $\sim$ is transitive and is an equivalence
relation.
\end{proof}

We now describe our method for partitioning $G$ into smaller blocks
for some fixed $(c,\alpha,\epsilon)$.
\begin{itemize}
\item Two $(c,\alpha,\epsilon)$-bad points $u,u'$ are in the same
block if and only if $u\sim u'$.

\item A $(c,\alpha,\epsilon)$-good vertex is in the same block as
any bad point it is adjacent to.

\item A $(c,\alpha,\epsilon)$-good vertex not adjacent to any bad
point forms a separate block
\end{itemize}
By Lemma \ref{l:equivalence_relation} the first point defines a
partition of the $(c,\alpha,\epsilon)$-bad vertices.  If a good
vertex $v$ is adjacent to bad vertices $u_1$ and $u_2$ then
$u_1,v,u_2$ has no two consecutive good points so $u_1\sim u_2$ and
hence good points are assigned to exactly one block.  Hence this
defines a partition of $G$ into blocks whose boundaries are all
$(c,\alpha,\epsilon)$-good.  We will abuse notation and let $\sim$
denote the equivalence relation on all $G$ for this partition.

\begin{lemma}\label{l:pathCuts}
Suppose that $G$ satisfies equation \eqref{eq_thm_hypothesis}.  Then
for any $0<L<\infty$ there exists $(c,\alpha,\epsilon)$ such that
every self-avoiding  path $u_1,u_2,\ldots,u_{L\log n}$ contains two
consecutive $(c,\alpha,\epsilon)$-good vertices $u_i,u_{i+1}$.
\end{lemma}

\begin{proof}
We can assume that $L \leq a$ and set $\epsilon=\frac{3\delta}{L}$.
 Then since $\sum_{i=1}^{L\log n} \varphi_\alpha(u_i) < \delta\log n$
at most $\frac{L}{3}\log n$ of the $u_i$ have
$\varphi_\alpha(u_i)\geq \epsilon$.  If $c=\frac{\epsilon}{\alpha}$
then if $\varphi_\alpha(u_i)<\epsilon$ then
\[
\hbox{deg}(u_i)=\sum_{u: (u,u_i)\in E} \alpha^{d(u,u_i)-1} \leq
\frac1{\alpha} \varphi_\alpha(u_i) < c
\]
so $u_i$ is $(c,\alpha,\epsilon)$-good.  Since the path
$u_1,u_2,\ldots,u_{L\log n}$ contains at least $\frac23 L \log n$
$(c,\alpha,\epsilon)$-good vertices it must contain two consecutive
good vertices.
\end{proof}

The following corollary is immediate from the definition of the
equivalence relation.

\begin{corollary}\label{c:pathCuts}
Suppose that $G$ satisfies equation \eqref{eq_thm_hypothesis}.  Then
for any $0<L<\infty$ there exists $(c,\alpha,\epsilon)$ such that if
$u\sim v$ then $d(u,v)<L\log n$.
\end{corollary}

Our next step is to define a partition of the graph into blocks
whose boundaries are good vertices and such that each block is
either a tree or a tree plus some bounded number of edges. The
decomposition into blocks relies on the following combinatorial
lemma.

\begin{lemma}\label{l:blockConstruction}
Consider a graph $G=(V,E)$ where $V$ is the disjoint union of $V_G$
and $V_B$.  Assume further that for all $v \in V$ it holds that
$t(v, a \log n) \leq t$ and that  every self avoiding path
$u_1,\ldots,u_{L \log n}$ contains two consecutive elements in
$V_G$, where $(20 t + 2) L < a$. Then we can partition $G$ into
blocks $\{V_j\}$ such there is at most one edge between any two
blocks. Moreover, for all $j$, the diameter of $V_j$ is less than
$(20t+2) L\log n$, it holds that $\partial V_j \subset V_G$, and
$V_j$ satisfies one of the following
\begin{itemize}
\item It is a tree.
\item There exist vertices $w_i$ and disjoint subsets $U_i\subset V_j$
such that each $U_i$ is a tree of depth at most $2L\log n$, $V_j =
\cup_i U_i$ and $w_i \in U_i$, there are no edges between $U_i-w_i$
and $V_j-U_i$. Furthermore the distance between $\partial V_j$ and
$W_j=\cup_i w_i$ is at least $L \log n$ and the subgraph $W_j$ has
$|W_j|\leq 20t L\log n$ and largest degree at most $2t$.
\end{itemize}

\end{lemma}

\begin{corollary}\label{c:blockConstruction}
Suppose that $G$ satisfies equation \eqref{eq_thm_hypothesis}.  Then
there exists $0<L <\infty$ and $(c,\alpha,\epsilon)$ such that we
can partition $G$ into blocks $\{V_j\}$ such there is at most one
edge between any two blocks. Moreover, for all $j$, the
diameter of $V_j$ is less than $(20t+2) L\log n$, it holds that
$\partial V_j \subset V_G$, and $V_j$ satisfies one of the following
\begin{itemize}
\item It is a tree.
\item There exist vertices $w_i$ and disjoint subsets $U_i\subset V_j$
such that each $U_i$ is a tree of depth at most $2L\log n$, $V_j =
\cup_i U_i$ and $w_i \in U_i$, there are no edges between $U_i-w_i$
and $V_j-U_i$. Furthermore the distance between $\partial V_j$ and
$W_j=\cup_i w_i$ is at least $L \log n$ and the subgraph $W_j$ has
$|W_j|\leq 20t L\log n$ and largest degree at most $2t$.
\end{itemize}

\end{corollary}

\begin{proof}
Letting $V_G$ be the set of good vertices and $V_B$ the set of bad vertices,
the proof of the corollary follows from Lemma~\ref{l:blockConstruction} by
taking $L$ such that $(20t+2) L<a$ and choosing $(c,\alpha,\epsilon)$
according to Corollary \ref{c:pathCuts}.
\end{proof}
We now prove Lemma~\ref{l:blockConstruction}.

\begin{proof}
The first step of the proof will be the construction of $W = \cup
W_j \subset V$. Beginning with $W$ as the empty set we can add to
$W$ in three ways:
\begin{itemize}
\item If $u_1,u_2,\ldots,u_m$ is a self-avoiding path of vertices in $V-W$
such that $u_1$ and $u_m$ are adjacent and $3 \leq m< 5 L \log n$ then
add $\{u_1,u_2,\ldots,u_m\}$ to $W$.

\item If $u_1,u_2,\ldots,u_m$ is a self-avoiding path in $V-W$
such that both $u_1$ and $u_m$ are adjacent to $W$ and $2\leq m< 5L\log
n$ then add $\{u_1,u_2,\ldots,u_m\}$ to $W$.

\item If $u_1$ is adjacent to two vertices in $W$
then add $\{u_1\}$ to $W$.
\end{itemize}
The construction of $W$ ends when no more additions are possible.
\begin{claim}
$W$ does not depend on the order of the additions.
\end{claim}
\begin{proof}
Note that if $W'$ and $W''$ are two different $W$'s obtained for different
order of additions then one may add all elements in $W' \setminus W''$ to $W'$
and vice-versa.
\end{proof}

\begin{claim}
At each stage of the construction no connected component $W_j$ of
$W$ is a tree; each connected component $W_j$ of $W$ has
\[
|W_j| \leq (10L t(W_j) - 5L)\log n,
\]
where $t(W_j)$ is the
tree excess of $W_j$.
\end{claim}

\begin{proof}
We split the additions into three cases.  If $u_1,u_2,\ldots,u_m$ is
not adjacent to any component of $W$ then this creates a new
component $W_{new}$ of $W$. This must be achieved by an addition of
the first type. The new component must contain a loop and have tree
excess at least $1$ and $|W_{new}|$ is less than $5L\log n$ which is
less than $(10L t(W_{new}) - 5L)\log n$.

Next suppose that an
addition $u_1,u_2,\ldots,u_m$ is adjacent to exactly one existing
component  $W_{old}$ of $W$. Then the addition forms a new component
$W_{new}$ which contains a new loop so $t(W_{new}) \geq
t(W_{old})+1$. On the other hand
\[
|W_{new}| \leq
(10L t(W_{old}) - 5L +
5L)\log n \leq (10L t(W_{new}) - 5L)\log n.
\]

Finally the addition
$u_1,u_2,\ldots,u_m$ may be adjacent to two or more components
$W_1,\ldots,W_k$ of $W$ and so forms one new component $W_{new}$
from these.  Then $t(W_{new}) \geq \sum_{j=1}^l t(W_j)$ and
\[
|W_{new}| \leq
5L\log n +\sum |W_j| \leq (10L
t(W_{new}) -5L)\log n.
\]
\end{proof}

\begin{claim}
When the construction of $W$ is completed, each component $W_j$ of $W$ is
of size at most $20t L\log n$ and tree excess at most $t$. The distance
between two components of $W$ is at least $5 L \log n$. All the degrees in
$W$ are bounded between $1$ and $2t$.
\end{claim}

\begin{proof}
We have seen that at each of the additions the tree excess of a
component increases by at least one. Suppose one of the components
of $W$ satisfies $|W_j| > 20 t L \log n$. If at some point in the
construction the maximum diameter of a component is $D$ then after
an addition the new maximum diameter is at most $2D+5L\log n$. So at
some point in the construction there must have been a component
$W_j$ with
\[
(10t-\frac52)L\log n \leq |W_j| \leq 20t L\log
n.
\]
Let $v\in W_j$.  Then $W_j\subset B(v,20t L\log n)$ so
$t(W_j) \leq t(v,20t L\log n)\leq t$.  Then
\[
|W_j| < (10L t(W_{j}) - 5L)\log n \leq (10t -5)L\log n,
\]
which is a contradiction.  Hence every component of $W$
has size at most $20t L\log n$ and tree excess at most
$t$.
By construction all components are separated by distance at least $5L\log
n$. Since the tree excess is at most $t$ and by construction
$W$ has no leaves the largest degree is at most $2t$.
\end{proof}

As in Lemma~\ref{l:equivalence_relation} for $u,u' \in V_B$ we write
$u \sim u'$ if there is a path connecting $u$ to $u'$ with no two
consecutive vertices belonging to $V_G$. For each component $W_j$ of
$W$ we define $V_j$ as
\[
V_j:=\{u\in V:\exists u'\in V,u\sim u', d(u',W_j)\leq L\}
\]
By construction $W_j\subset V_j$ and if $d(u,W_j)\leq L\log n$ then
$u\in V_j$ while if $d(u,w_j)\geq 2L\log n$ then by Corollary
\ref{c:pathCuts} $u\not\in V_j$.  It follows that the components
$V_j$ are disjoint and are not adjacent.  We will show that the
components satisfy the  hypothesis of the lemma.

Suppose that there exist two self-avoiding paths
$u_0,u_1,\ldots,u_l$ and $v_0,v_1,\ldots,v_m$ with $u_l=v_m$,
$u_0,v_0\in W_j$ and $u_1,\ldots,u_l,v_1,\ldots,v_m \in V_j-W_j$
which are not identical, (i.e. for some $i$, $u_i\neq v_i$).  If
$l+m\leq 5 L \log n$ then $u_0,u_1,\ldots,u_l,v_0,v_1,\ldots,v_m$
must contain a loop of length less than $5L\log n$ which could be
added to $W$ contradicting our assumption.  So without loss of
generality $l\geq \frac52 L \log n$.  Then there exists $u'$ with
$u'\sim u_{\frac52L\log n}$ and $d(u',W_j)\leq L\log n$.  Then there
exists a path in the equivalence class of $u'$ from $u_{\frac52L\log
n}$ to $u'$ with length at most $L\log n$.  Since $d(u',w)\leq L$
for some $w\in W$ there also exists a path from $u'$ to $w$ in
$\{u:d(u,W)\leq L\}\subset V_j$ with length at most $L\log n$.
Combining these paths there is a path from $u_{\frac52L\log n}$ to
$w$ in $V_j$ of length at most $2L\log n$.  Combining this path with
$u_0,u_1,\ldots,u_{\frac52L\log n}$ we must have a loop of length at
most $\frac92 L\log n$.  But this could be an addition to $W$ which
is a contradiction.  Hence for each $u\in V_j-W_j$ there is a unique
self-avoiding path from $u$ to $W_j$ in $V_j-W_j$.  It follows that
we can partition $V_j$ into $\{U_i\}$ as required.

Those points in $V_B$ that are not in some $V_j$
can be placed in blocks according to
their equivalence class from the relation $\sim$.  All such extra
blocks are trees of maximum diameter $L\log n$. Finally, vertices  $v \in V_G$
belong to the block defined by $u \in V_B$ if $(u,v)$ is an edge $E$ and
if no such edge exists $v$ is a seperate block.
\end{proof}

\subsection{Block Relaxation Times}

\subsubsection{Colouring Model}

\begin{lemma}\label{l:blockrelaxcolour}
Suppose that $G$ satisfies equation \eqref{eq_thm_hypothesis}.  For
sufficiently large $q$ the relaxation times of the Glauber dynamics
on each of the blocks constructed in Lemma \ref{l:blockConstruction}
is bounded by $n^C$.
\end{lemma}

\begin{proof}
In the blocks $V_j$ which are trees any path is of length at most
$20t L \log n$ so

\[
m(V_j,v)\leq
\frac1{\alpha}m_\alpha(V_j,20t L \log n)\leq (1+\frac{20t
L}{a})\frac{\delta}{\alpha}\log n.
\]
By Theorem
\ref{theorem_tree_mixing} and Lemma \ref{l:treeCorrDecay} the
relaxation time is bounded by $n^C$.

Now consider a block $V_j$ of the second type.  We divide $V_j$ into
its sub-blocks $U_i$.  Each $U_i$ is a tree and every
$v\in\partial^+_{V_j} U_i$ is $(c,\alpha,\epsilon)$-good.  Any path
in $U_i$ has length at most $2L \log n$ so
\[
m(U_i,w_i)\leq
\frac1{\alpha}m_\alpha(U_i,2L \log n)\leq
(1+\frac{2L}{a})\frac{\delta}{\alpha}\log n.
\]
Then by Theorem
\ref{theorem_tree_mixing} and Lemma \ref{l:treeCorrDecay} the
relaxation time of the Glauber dynamics on each $U_i$ is bounded by
$n^{C'}$.

In Lemma \ref{l:treeCorrDecay} take $q$ to be large enough so that
$\log\lambda < -4/L$.  Then for $w_i\in W_j$,
\begin{align}\label{e:blockDecayCorr}
\sup_{\sigma(\partial^+_{V_j} U_i)} \sup_{x,y\in \mathcal{C}}
\frac{P_{U_i \cup \partial^+_{V_j}
U_i}(\sigma(w_i)=x|\sigma(\partial^+_{V_j} U_i))}{P_{U_i \cup
\partial^+_{V_j}}(\sigma(w_i)=y|\sigma(\partial^+_{V_j} U_i))} &\leq
\exp(\sum_{v\in
\partial^+_{V_j} U_i} \lambda^{d(w_i,v)})\\
& \leq \exp(\sum_{v\in
\partial^+_{V_j} U_i} \lambda^{L\log n})\\
&\leq \exp(n^{-3})
\end{align}


so $P(\sigma(w_i)=x|\sigma(\partial^+_{V_j} U_i))\geq
q^{-1}\exp(-n^{-3})$. Then by Lemmas \ref{lemma_block_dynamics} and
\ref{lemma_small_graph_mixing} the relaxation time of the block
dynamics with blocks $\{U_i\}$ is bounded by $n^{C''}$. Then by
Proposition 3.4 of \cite{Martinelli:99} we have that the relaxation
time of the Glauber dynamics on $V_j$ is bounded by $n^C$.
\end{proof}

\subsubsection{Hardcore Model}

\begin{lemma}\label{l:hardblockrelax}
Suppose that $G$ satisfies equation \eqref{eq_thm_hypothesis}.  For
sufficiently small $\beta$ the relaxation times of the Glauber
dynamics on each of the blocks constructed in Lemma
\ref{l:blockConstruction} is bounded by $n^C$.
\end{lemma}

\begin{proof}
In the blocks $V_j$ which are trees, any path is of length at most
$20t L \log n$ so
\[
m(V_j,v)\leq \frac1{\alpha}m_\alpha(V_j,20t L
\log n)\leq (1+\frac{20t L}{a})\frac{\delta}{\alpha}\log n.
\]
By
Theorem \ref{theorem_tree_mixing_hardcore} the relaxation time is
bounded by $n^C$.

Now consider a block $V_j$ of the second type.  By Lemmas
\ref{lemma_block_dynamics} and \ref{lemma_hard_small_mixing} the
relaxation time of the block dynamics with blocks $\{U_i\}$ is
bounded by $n^{C''}$. Then by Proposition 3.4 of
\cite{Martinelli:99} we have that the relaxation time of the Glauber
dynamics on $V_j$ is bounded by $n^C$.
\end{proof}

\subsubsection{Soft Constraints}
\begin{lemma}\label{l:hardblockrelax}
Suppose that $G$ satisfies equation \eqref{eq_thm_hypothesis}.  For
small $\|H\|$ the relaxation times of the Glauber dynamics on each
of the blocks constructed in Lemma \ref{l:blockConstruction} is
bounded by $n^C$.
\end{lemma}

\begin{proof}
In the blocks $V_j$ which are trees any path is of length at most
$20t L \log n$ so
\[
m(V_j,v)\leq \frac1{\alpha}m_\alpha(V_j,20t L
\log n)\leq (1+\frac{20t L}{a})\frac{\delta}{\alpha}\log n.
\]
By
Theorem \ref{thm_relax_softcore} the relaxation time is bounded by
$n^C$.

Now consider a block $V_j$ of the second type. Let $V_j'$ be the
block obtained by removing each of the edges in the skeleton $W_j$
and let $\tau'$ be the relaxation time on $V_j'$. In the proof of
Lemma~\ref{thm_relax_softcore} we showed that removing an edge
affects the relaxation time by a factor of at most $\exp(4\| H\|)$
so $\tau\leq n^{80 \| H \| t}\tau'$. In $V_j'$ each of the trees
$U_i$ is separated so $\tau'$ is simply the maximum of the
relaxation times of the $U_i$.  By Theorem \ref{thm_relax_softcore}
the relaxation time is bounded by $n^{C'}$ so each of the $U_i$ are
bounded by $n^{C'}$ so $\tau\leq n^C$.
\end{proof}

\subsection{Mixing time of block dynamics}

We use the partition from Lemma \ref{l:blockConstruction} as blocks
for the block dynamics of the Glauber dynamics.  We use the method
of path coupling to bound the mixing time of the block dynamics. Let
$d_{H}$ denote the hamming distance of two distributions. Suppose
that $T\subset V$ is a tree, let $v\in
\partial^+ T$ be $(c,\alpha,\epsilon)$-good and let $\eta,\eta'$ be
two boundary conditions on $\partial^+ V_j$ which differ only at $v$
and suppose that $\rho$ is the only vertex in $T$ adjacent to $v$.
We must couple two states $\sigma(T),\sigma'(T)$ so that they are
distributed as $Q$ and $Q'$ respectively where
$Q(\sigma(T))=P(\sigma(T)|\eta)$ and
$Q'(\sigma'(T))=P(\sigma'(T)|\eta')$. This can be done as follows.
Root $T$ at $\rho$ and let $\overleftarrow{u}$ denote the parent of
$u\in T$. First couple $\sigma(\rho)$ and $\sigma'(\rho)$ according
to their marginal distributions $P(\sigma(\rho)|\eta)$ and
$Q'(\sigma'(\rho)|\eta')$ so as to minimize their total variation
distance. Proceed inductively down the tree by coupling $\sigma(u)$
and $\sigma'(u)$ according to
$P(\sigma(u)|\eta,\sigma(\overleftarrow{u}))$ and
$P(\sigma'(u)|\eta,\sigma'(\overleftarrow{u}))$ so as to minimize
the total variation distance.  When
$\sigma(\overleftarrow{u})=\sigma'(\overleftarrow{u})$ then
$\sigma(u)=\sigma'(u)$.  We will show that we can bound the expected
hamming distance of these coupled distributions.

\subsubsection{Colouring Model}

\begin{lemma}\label{l:treeCoupleColour}
Let $T$ be a tree such that $\psi(u)= \sum_{w\in
\partial^+ T} \lambda^{d(w,u)}<\epsilon$ for all $u\in T$.  If $\delta>0$ then for
some sufficiently large $q=q(\delta,\epsilon,\lambda)$, the above
coupling has
\[
E d_H(\sigma(T),\sigma'(T))\leq \delta.
\]
\end{lemma}

\begin{proof}
Let $\gamma>0$ such that $\varphi_\gamma(v)<\delta$.  For all $u\in
T$ we have that $\#\{w\in V-T:(w,u)\in E\}\leq \epsilon/\lambda$. By
Lemma \ref{l:treeCorrDecay} we choose $q$ large enough so that for
each $u\in T$ and $x\in \mathcal{C}$, $P(\sigma(u)=x|\eta)<
\gamma/2$. Then
\[
d_{TV}( P(\sigma(u)=\cdot|\eta,\sigma(\overleftarrow{u})),
P(\sigma(u)=\cdot|\eta,\sigma'(\overleftarrow{u}))) \leq 2 \max_x
P(\sigma(u)=x|\eta) < \gamma.
\]
So given that $\sigma(\overleftarrow{u})$ and
$\sigma'(\overleftarrow{u})$ disagree then $\sigma(u)$ and
$\sigma'(u)$ disagree with probability at most $\gamma$.  It follows
that the probability that $\sigma(u)$ and $\sigma'(u)$ disagree is
at most $\gamma^{d(u,v)}$ and so $E d_H(\sigma(T),\sigma'(T)) \leq
\sum_{u\in T} \gamma^{d(u,v)} \leq\varphi_\gamma(v)<\delta$ as
required.
\end{proof}

\begin{lemma}\label{l:blockCouple}
Let $V_j$ be a block constructed from Lemma
\ref{l:blockConstruction}.  If $v\in\partial^+ V_j$ and $\eta,\eta'$
are boundary conditions on $\partial^+ V_j$ which differ only at $v$
then for sufficiently large $q=q(a,\alpha, t,\delta)$ we can couple
colourings $\sigma(V_j),\sigma'(V_j)$ distributed as
$P(\sigma(V_j)|\eta),P(\sigma'(V_j)|\eta')$ respectively so that
\[
E d_H(\sigma(V_j),\sigma'(V_j))\leq \delta.
\]
\end{lemma}

\begin{proof}
The case when $V_j$ is a tree follows by Lemma
\ref{l:treeCoupleColour} so we consider the blocks of the second
type.  Let $v$ be adjacent to $U_i$.  If $\sigma^1(W_j)$ and
$\sigma^2(W_j)$ are two colourings of $W_j$ then by equation
(\ref{e:blockDecayCorr})
\[
\frac{P(\sigma^1(W_j)|\eta)}{P(\sigma^2(W_j)|\eta)} = \prod_i
\frac{P(\sigma^1(w_i)|\eta(\partial^+_{V_j}
U_i))}{P(\sigma^2(w_i)|\eta(\partial^+_{V_j} U_i))} \leq \prod_i
\exp(n^{-3}) \leq \exp(n^{-2})
\]
and so the total variation distance between $P(\sigma(W_j)|\eta)$
and the free measure on colourings on $W_j$ is $O(n^{-2})$.  It
follows that we can couple $\sigma(W_j)$ and $\sigma'(W_j)$ so that
they agree with probability $1-O(n^{-2})$.  On the event they
disagree there are at most $|V_j|\leq n$ disagreements so this event
contributes $O(n^{-1})$ disagreements to the expected value.  So now
on the event that $\sigma(W_j)=\sigma'(W_j)$ for all $k\neq i$ we
can set $\sigma(U_k-\{w_k\})=\sigma'(U_k-\{w_k\})$ since they have
the same boundary conditions. This just leaves $\sigma(U_i-\{w_i\})$
and $\sigma'(U_i-\{w_i\})$ to be coupled. Now $U_i-\{w_i\}$ is a
tree which has every boundary vertex $(c,\alpha,\epsilon)$-good
except perhaps $w_i$.  Then repeating the argument of Corollary
\ref{c:treeCorrDecay} we have that when $\lambda=\alpha^2$
\[
\psi(u)  \leq \lambda + \sum_{u'\in
\partial^+ U_i-\{w_i\}} \lambda^{d(u',u)}\leq \lambda + \epsilon.
\]
Applying Lemma \ref{l:treeCoupleColour} to $U_i-\{w_i\}$ completes
the result.
\end{proof}

\begin{lemma}\label{l:blockdynamicsrelax}
For large enough $q$ the relaxation time of the block dynamics with
blocks $\{V_j\}$ from Lemma \ref{l:blockConstruction} is $O(n)$.
\end{lemma}

\begin{proof}
Choose $q$ large enough so that in Lemma \ref{l:blockCouple} we can
take $\delta< c$.  By the method of path coupling described in
Section \ref{s:pathCouple} it is sufficient to show that if
$\sigma_0,\sigma_0'$ are two colourings with
$d_H(\sigma_0,\sigma_0')=1$ differing only at $v$ then we can couple
one step of the block dynamics so that the new pair
$\sigma_1,\sigma_1'$ has
\[
E d(\sigma_1,\sigma_1') \leq 1-\beta/n
\]
for some $\beta>0$.  Let $K$ be the number of blocks.  We couple
them as follows.  If the block $V_j$ chosen by the block dynamics
contains $v$ then we set $\sigma(V_j)=\sigma'(V_j)$ and have
$d(\sigma_1,\sigma_1')=1$. If the block chosen is adjacent to $v$
then we couple $V_j$ according to Lemma \ref{l:blockCouple}.  The
expected number of new disagreements is at most $\delta$.  If $V_j$
neither contains nor is adjacent to $v$ then we set
$\sigma(V_j)=\sigma'(V_j)$ and the number of disagreements does not
change.  Now if $v$ is adjacent to some blocks $V_j$ it must be in
the boundary and so therefore must be $(c,\alpha,\epsilon)$-good.
Since it has degree at most $c$ it is adjacent to at most $c$ blocks
so
\[
E d(\sigma_1,\sigma_1') \leq 1 - \frac1{K}+c\frac{\delta}{K} \leq
1-\beta/n
\]
where $\beta=1-c\delta$ which completes the proof.
\end{proof}

\subsubsection{Hardcore Model}

\begin{lemma}\label{l:treeCoupleHard}
Let $T$ be a tree such that $\psi(u)= \sum_{w\in
\partial^+ T} \lambda^{d(w,u)}<\epsilon$ for all $u\in T$.  If $\delta>0$ then there exists
$\beta^*=\beta^*(\delta,\lambda,\epsilon)$ such that if
$\beta<\beta^*$, the above coupling has
\[
E d_H(\sigma(T),\sigma'(T))\leq \delta.
\]
\end{lemma}

\begin{proof}
Let $\gamma>0$ such that $\varphi_\gamma(v)<\delta$.  We can choose
$\beta$ small enough so that $\frac{e^\beta}{1+\beta}<\gamma$.  For
all $u\in T$, $P(\sigma(u)=1|\eta)\leq
P(\sigma(u)=1|\sigma(V-\{u\})\equiv 0) =
\frac{e^\beta}{1+\beta}<\gamma$. Then
\[
d_{TV}( P(\sigma(u)=\cdot|\eta,\sigma(\overleftarrow{u})),
P(\sigma(u)=\cdot|\eta,\sigma'(\overleftarrow{u}))) \leq
P(\sigma(u)=1|\eta,\sigma(\overleftarrow{u}))-
P(\sigma(u)=1|\eta,\sigma'(\overleftarrow{u})) < \gamma.
\]
So given that $\sigma(\overleftarrow{u})$ and
$\sigma'(\overleftarrow{u})$ disagree then $\sigma(u)$ and
$\sigma'(u)$ disagree with probability at most $\gamma$.  It follows
that the probability that $\sigma(u)$ and $\sigma'(u)$ disagree is
at most $\gamma^{d(u,v)}$ and so $E d_H(\sigma(T),\sigma'(T)) \leq
\sum_{u\in T} \gamma^{d(u,v)} \leq\varphi_\gamma(v)<\delta$ as
required.
\end{proof}

The following results follow similarly to the colouring model.

\begin{lemma}\label{l:hardblockCouple}
Let $V_j$ be a block constructed from Lemma
\ref{l:blockConstruction}.  For $\delta>0$ there exists
$\beta^*=\beta^*(a,\alpha, t,\delta)$ such that for $\beta<\beta^*$
if $v\in\partial^+ V_j$ and $\eta,\eta'$ are boundary conditions on
$\partial^+ V_j$ which differ only at $v$ then we can couple states
$\sigma(V_j),\sigma'(V_j)$ distributed as
$P(\sigma(V_j)|\eta),P(\sigma'(V_j)|\eta')$ respectively so that
\[
E d_H(\sigma(V_j),\sigma'(V_j))\leq \delta.
\]
\end{lemma}

\begin{lemma}\label{l:blockdynamicsrelaxhard}
There exists $\beta^*=\beta^*(a,\alpha, t,\delta)$ such that for
$\beta<\beta^*$ the relaxation time of the block dynamics with
blocks $\{V_j\}$ from Lemma \ref{l:blockConstruction} is $O(n)$.
\end{lemma}

\subsubsection{Soft Constraints Model}

\begin{lemma}\label{l:treeCoupleSoft}
Let $T$ be a tree such that $\psi(u)= \sum_{w\in
\partial^+ T} \lambda^{d(w,u)}<\epsilon$ for all $u\in T$.  If $\delta>0$ then there exists
$H^*=H^*(\delta,\lambda,\epsilon) > 0$ such that if $\|H\|<H^*$, the
above coupling has
\[
E d_H(\sigma(T),\sigma'(T))\leq \delta.
\]
\end{lemma}

\begin{proof}
Let $\gamma>0$ such that $\varphi_\gamma(v)<\delta$. Repeating the
argument of Lemma~\ref{l:treeCorrDecayHard} we can choose $\|H\|$
small enough so that
\[
d_{TV}( P(\sigma(u)=\cdot|\eta,\sigma(\overleftarrow{u})),
P(\sigma(u)=\cdot|\eta,\sigma'(\overleftarrow{u}))) < \gamma.
\]
The remainder of the proof follows similarly from
Lemma~\ref{l:treeCoupleHard}.
\end{proof}

The following results follow similarly from the colouring model.

\begin{lemma}\label{l:SoftblockCouple}
Let $V_j$ be a block constructed from Lemma
\ref{l:blockConstruction}.  For $\delta>0$ there exists
$H^*=H^*(a,\alpha, t,\delta)$ such that for $\|H\|<H^*$ if
$v\in\partial^+ V_j$ and $\eta,\eta'$ are boundary conditions on
$\partial^+ V_j$ which differ only at $v$ then we can couple states
$\sigma(V_j),\sigma'(V_j)$ distributed as
$P(\sigma(V_j)|\eta),P(\sigma'(V_j)|\eta')$ respectively so that
\[
E d_H(\sigma(V_j),\sigma'(V_j))\leq \delta.
\]
\end{lemma}

\begin{lemma}\label{l:blockdynamicsrelaxhard}
There exists $H^*=H^*(a,\alpha, t,\delta)$ such that for $\|H\|<H^*$
the relaxation time of the block dynamics with blocks $\{V_j\}$ from
Lemma \ref{l:blockConstruction} is $O(n)$.
\end{lemma}

\subsection{Main Results}
The main results now follows easily using the block dynamics
approach of Proposition 3.4 of \cite{Martinelli:99}.

\begin{proof}(Theorem \ref{t:main})
For large enough $q$, by Lemma \ref{l:blockdynamicsrelax} the
relaxation time of the block dynamics of the Glauber dynamics on $G$
with blocks $\{V_j\}$ from Lemma \ref{l:blockConstruction} is
$O(n)$. By Lemma \ref{l:blockrelaxcolour} the relaxation time of the
Glauber dynamics on each block is bounded by $n^{C'}$.  Then by
Proposition 3.4 of \cite{Martinelli:99} we have that the relaxation
time is $O(n^{C'+1})$.  There are at most $q^n$ colourings of $G$ so
$\log(1/\min_\sigma P(\sigma))\leq n\log q$ so the mixing time of
the Glauber dynamics is bounded by $O(n^{C'+2})$ which completes the
result.
\end{proof}

The proofs of Theorems \ref{t:mainHard} and \ref{t:mainSoft} follow
similarly.

\bibliographystyle{plain}
\bibliography{all,my}

\end{document}